\documentclass[final,1p,times]{elsarticle}

\usepackage{amssymb}
\usepackage{amsmath}

\usepackage{lineno}

\newcommand*\patchAmsMathEnvironmentForLineno[1]{%
  \expandafter\let\csname old#1\expandafter\endcsname\csname #1\endcsname
  \expandafter\let\csname oldend#1\expandafter\endcsname\csname end#1\endcsname
  \renewenvironment{#1}%
     {\linenomath\csname old#1\endcsname}%
     {\csname oldend#1\endcsname\endlinenomath}}%
\newcommand*\patchBothAmsMathEnvironmentsForLineno[1]{%
  \patchAmsMathEnvironmentForLineno{#1}%
  \patchAmsMathEnvironmentForLineno{#1*}}%
\AtBeginDocument{%
  \patchBothAmsMathEnvironmentsForLineno{equation}%
  \patchBothAmsMathEnvironmentsForLineno{align}%
  \patchBothAmsMathEnvironmentsForLineno{flalign}%
  \patchBothAmsMathEnvironmentsForLineno{alignat}%
  \patchBothAmsMathEnvironmentsForLineno{gather}%
  \patchBothAmsMathEnvironmentsForLineno{multline}%
}

\usepackage{lipsum}
\usepackage{amsfonts}
\usepackage{amsthm}
\usepackage{mathtools}
\usepackage{epstopdf}
\usepackage{pgfplots}
\usepackage{pgfplotstable}
\pgfplotsset{compat=newest}
\usepackage{hyperref}
\usepackage{amsopn}
\usepackage{nicematrix}

\usepackage{float}
\usepackage{subcaption}
\usepackage{bm}
\usepackage{tikz}
\usepackage{xcolor}
\usepackage{enumitem}
\usepackage{listings}
\usepackage{url}
\usepackage[normalem]{ulem}

\newif\ifshowchanges
\showchangesfalse
\colorlet{revcolor}{blue!60!black}
\ifshowchanges
  \newcommand{\rev}[1]{{\color{revcolor}#1}}
  \newcommand{\del}[1]{{\color{red}\sout{#1}}}
  \newenvironment{revblock}{\color{revcolor}}{}
\else
  \newcommand{\rev}[1]{#1}
  \newcommand{\del}[1]{}
  \newenvironment{revblock}{}{}
\fi

\usetikzlibrary{arrows.meta,calc,
	tikzmark,pgfplots.groupplots, matrix,positioning,matrix}

\newtheorem{theorem}{Theorem}[section]
\newtheorem{proposition}[theorem]{Proposition}
\newtheorem{lemma}[theorem]{Lemma}

\newtheorem{definition}[theorem]{Definition}
\theoremstyle{definition}
\newtheorem{example}[theorem]{Example}
\newtheorem{remark}[theorem]{Remark}

\newtheorem{problem}[theorem]{Problem}

\newcommand{\C}{\mathbb{C}}
\renewcommand{\vec}[1]{\bm{#1}}

\journal{Journal of Computational and Applied Mathematics}

\begin{document}

\begin{frontmatter}



\title{Orthogonal Vector-Valued Rational Functions with Applications in Rational Approximation}


\author[label1]{Robbe Vermeiren}
\ead{robbe.vermeiren@kuleuven.be}
\affiliation[label1]{organization={KU Leuven},
            addressline={Celestijnenlaan 200A},
            city={Leuven},
            postcode={3001},
            country={Belgium}}
\begin{abstract}
We present two algorithms for constructing orthonormal bases of vector-valued rational functions with respect to a discrete inner product. \rev{Every pole is prescribed together with a tangential direction, which fixes the direction of the residue at that pole.} Building on the pencil formulation of the inverse eigenvalue problem for scalar rational functions by Van Buggenhout, Van Barel and Vandebril (2022), we show that, for functions with $k$ components, the recurrence relations are represented by a pencil of $k$-Hessenberg matrices, i.e., matrices with at most $k$ nonzero subdiagonals, \rev{and that the poles are the ratios of the entries on the $k$th subdiagonal. This extends the banded Hessenberg structure of orthogonal vector-valued polynomials to rational functions, and it replaces the semiseparable-plus-diagonal representation of Delvaux and Van Barel (2005), in which every pole belongs to a single component, by a pencil.} We derive an updating algorithm based on rotations and a Krylov-type algorithm \rev{that orthogonalizes a tangential rational Krylov subspace, and we compare their accuracy. Finally, we show that the lightning method and the AAA algorithm fit in our framework; for AAA, the barycentric weights follow from a singular value decomposition of a small block of the pencil.} The approximation of $\sqrt{z}$ on $[0,1]$ with exponentially clustered poles and sample points serves as a stress test for our methods.
\end{abstract}




\end{frontmatter}



\section{Introduction}
A well-studied problem in applied mathematics is finding an orthonormal basis of a finite-dimensional vector space $V_n$ of functions with respect to a discrete inner product $\langle \cdot, \cdot \rangle$. In many cases, this problem can be solved with a Krylov-type algorithm. \rev{This requires a nested sequence of spaces $V_1 \subset V_2 \subset \cdots \subset V_n$ with $\dim V_j = j$, such that $V_{j+1}$ is spanned by $V_j$ and the product of an element of $V_j$ with a simple function, e.g., with $z$ for polynomials. The next orthonormal basis function is then obtained by orthogonalizing this product against $V_j$.} The orthonormal basis functions satisfy recurrence relations, and collecting their coefficients in a \emph{recurrence matrix} turns the problem into an inverse eigenvalue problem (IEP).

For example, let $V_n$ be the space of polynomials of degree $< n$ and $\langle p, q \rangle = \sum_{i=1}^n |w_i|^2\overline{p(z_i)}q(z_i)$, with nonzero weights $w_i \in \C$ and pairwise distinct nodes $z_i \in \C$. The IEP is then to find a unitary matrix $Q$ and an upper Hessenberg recurrence matrix $H \in \C^{n\times n}$ such that
\begin{equation}\label{eq:polIEP}
    ZQ = QH,
\end{equation}
where $Z=\operatorname{diag}([z_i]_{i=1}^n)$ and the first column of $Q$ is an appropriate scaling of the vector containing the weights $w_i$. For nodes on the real line, $H$ is tridiagonal~\cite{GraggHarrod1984}, and for nodes on the complex unit circle, $H$ is a unitary Hessenberg matrix~\cite{AmmarHe1995,AmmarGraggReichel}. \rev{Two kinds of algorithms solve this IEP. Krylov-type algorithms run the Lanczos or Arnoldi process on $Z$, starting from the vector of weights, whereas updating algorithms add the nodes one at a time and restore the structure of $H$ with unitary similarity transformations~\cite{Rutishauser1963,GraggHarrod1984,Reichel1991}. For real nodes, Gragg and Harrod~\cite{GraggHarrod1984} showed that the Lanczos process is numerically unstable for this task, whereas the rotation-based updating algorithm of Rutishauser~\cite{Rutishauser1963} is stable.}

Van Barel and Bultheel generalized the IEP~\eqref{eq:polIEP} to orthonormal vector-valued polynomials, i.e., vectors with polynomials as components~\cite{VanBarelBultheel1992,BultheelVanBarel1995,VanBarelBultheel1995}. \rev{These polynomial vectors solve discrete linearized least-squares rational approximation problems~\cite{VanBarelBultheel1992,VanBarelBultheel1994,VanBarelBultheel1995}.} IEPs have also been formulated for Sobolev orthogonal polynomials~\cite{Vanbuggenhout2023}, whose inner product contains derivatives, and for multiple orthogonal polynomials~\cite{FaghihRinelliVanBaVandeVerme2026}, which are orthogonal with respect to several inner products \rev{and whose recurrence matrix is a banded Hessenberg matrix}.

For rational functions with prescribed poles $\zeta_i$, Van Barel, Fasino, Gemignani and Mastronardi~\cite{VanBFasinoGemignaniMastronardi2005} showed that the IEP takes the form $ZQ = Q(S+\Lambda_\zeta)$, where $S$ is a lower semiseparable matrix\footnote{\rev{A matrix $S \in \C^{n\times n}$ is called lower semiseparable if every submatrix that can be taken out of its lower triangular part (diagonal included), has rank at most 1~\cite{VandebrilVanbarelGolubMastronardi2005}.}} and $\Lambda_\zeta$ is a diagonal matrix containing the poles. Van Buggenhout, Van Barel and Vandebril~\cite{VanBuVanBaVand2022} proposed an alternative formulation of the IEP for rational functions,
\begin{equation}\label{eq:IEP}
    ZQK = QH,
\end{equation}
where $K$ and $H$ are upper Hessenberg matrices. We write them as the matrix pencil $(H,K)$ to stress the connection with generalized eigenvalue problems. \rev{The poles of the rational functions are the ratios $h_{i+1,i}/k_{i+1,i}$ of the subdiagonal entries of the pencil\footnote{Here, $h_{ij}$ and $k_{ij}$ denote the entries of $H$ and $K$ in row $i$ and column $j$.}, a property that is well known from rational Krylov methods~\cite{Ruhe1984,BerljafaGuttel2015}.} Van Buggenhout et al.~\cite{VanBuVanBaVand2022} present an updating algorithm, which operates directly on $(H,K)$ using sequences of rotations, and a Krylov-type algorithm, which is a special case of the rational Arnoldi algorithm~\cite{Ruhe1984}. An updating algorithm starts from the pencil $(H_j, K_j)$ for an inner product with $j$ nodes and weights, and updates it to the pencil $(H_{j+1}, K_{j+1})$ for the inner product with one additional node $z_{j+1}$ and weight $w_{j+1}$. \rev{In block rational Arnoldi decompositions~\cite{ElsworthGuttel2020}, which encode rational matrix-valued functions, the pencil is block upper Hessenberg and every pole belongs to a whole block, i.e., it is added to all components at once.}

In the present work, $V_n$ consists of vector-valued rational functions, also called rational function vectors. \rev{Each pole $\zeta_i$ comes with a prescribed tangential direction $\vec d_i$, as in the tangential rational Krylov subspaces that are used for model order reduction of multi-input, multi-output (MIMO) dynamical systems~\cite{GallivanVandendorpeVanDooren2004,DruskinSimonciniZaslavsky2014}. This direction fixes the residue of the vector-valued rational function at the pole up to a scalar factor, and thus determines how the pole enters the components. In particular, a pole can belong to a single component ($\vec d_i = \vec e_l$). Delvaux and Van Barel~\cite{DelvauxVanBa2005} studied this special case, with the poles added one per component in turn, and extended the semiseparable-plus-diagonal structure of~\cite{VanBFasinoGemignaniMastronardi2005} to it.}

For vectors of length $k$, we show that the IEP retains the form~\eqref{eq:IEP}, where $(H,K)$ is now a pencil of two $k$-Hessenberg matrices, i.e., $h_{ij} = k_{ij} = 0$ for $i > j+k$. \rev{The poles are the ratios of the entries on the $k$th subdiagonal of $(H,K)$, and the tangential directions are imposed by the choice of the recurrence coefficients. When all poles are at infinity, we recover the IEP~\eqref{eq:polIEP} for vector-valued polynomials with its generalized Hessenberg matrix~\cite{BultheelVanBarel1995,VanBarelBultheel1995}. To the best of our knowledge, this correspondence between $k$-Hessenberg pencils and tangential rational Krylov subspaces with one direction per pole has not been made explicit before. It extends the banded Hessenberg structure of orthogonal vector-valued polynomials~\cite{BultheelVanBarel1995} to rational functions, and it replaces the semiseparable-plus-diagonal representation of~\cite{DelvauxVanBa2005} by a pencil in which the poles and the tangential directions are read off from, and imposed on, individual entries}; as seen from our algorithms, this pencil is more convenient to work with. Similarly as in~\cite{VanBuVanBaVand2022}, we derive two algorithms that solve this IEP. The updating algorithm restores the structure of the pencil with rotations; \rev{it chases the fill-in by swapping the eigenvalues of $2\times 2$ triangular pencils~\cite{VanDooren1981}, the basic operation of pole-swapping algorithms~\cite{CampsMeerbergenVandebril2019a,CampsEtAl2020a}.} The Krylov-type algorithm orthogonalizes a tangential rational Krylov subspace. \rev{As in the polynomial case~\cite{GraggHarrod1984}, the two approaches behave differently in finite precision: in the experiments of Section~\ref{sec:numexp}, both compute a numerically unitary matrix $Q$, the updating algorithm yields functions that are closer to orthonormal, and the Krylov-type algorithm reproduces the poles and the recurrence relations more accurately. For ease of exposition, we present the theory and the algorithms for $k=2$; the generalization to $k>2$ is straightforward. Functions with more than two components arise, for example, in the set-valued, block and tangential variants of the AAA algorithm~\cite{LietaertEtAl2022,GoseaGuttel2021,AumannEtAl2023}.}


Rational functions are fundamental to several nonlinear approximation problems, for example in model reduction~\cite{BultheelDeMoor2000}, system identification~\cite{NinnessGustafsson1997} and PDE problems~\cite{BrubeckTrefethen2022,Trefethen2020}. \rev{Two common strategies make these problems tractable: fixing the poles in advance, which leads to a linear least-squares problem~\cite{GopalTrefethen2019,Baddoo2020}, and linearizing the problem by fitting $n(z) \approx f(z)d(z)$ instead of $f(z) \approx n(z)/d(z)$~\cite{Levy1959,VanBarelBultheel1994}. Both strategies fit in our framework if we collect the numerator and denominator of an approximant $r(z) = n(z)/d(z)$ in the vector $[n(z),d(z)]^T$. A pole $\zeta$ of this vector with tangential direction $[a,b]^T$ generically imposes the value $a/b$ on $r$ at that pole, so that the directions $\vec e_1$, $\vec e_2$ and $[f(\zeta),1]^T$ prescribe a pole, a zero and an interpolation condition of $r$ at $\zeta$, respectively. Moreover, for a suitable choice of the weight vectors, the norm induced by our inner product equals the linearized least-squares error. In Section~\ref{sec:application}, we show that the lightning method~\cite{GopalTrefethen2019,HerremansHuybrechsTrefethen2023}, which fixes the poles of $r$ in advance, clustered near the singularities of $f$, and the AAA algorithm~\cite{NakatsukasaSeteTrefethen2018,NakatsukasaTrefethen2026}, which adaptively selects interpolation points, are two instances of this framework. They differ only in the choice of the poles, the tangential directions and the nodes, and in the normalization that excludes the trivial solution. For AAA, the barycentric weights follow from a singular value decomposition of a small block of the pencil, without forming the Loewner matrix used in AAA.} We apply both methods to the approximation of $f(z)=\sqrt{z}$ on $[0,1]$. The poles span many orders of magnitude and the sample points are exponentially graded towards the singularity at $z=0$, so this experiment serves as a stress test for our algorithms.

    The paper is organized as follows.
    Section~\ref{sec:problemformulation} introduces the theoretical framework needed to formulate the IEP.
    Section~\ref{sec:updating} describes the updating algorithm, while Section~\ref{sec:krylov} handles the rational Krylov-type method.
    In Section~\ref{sec:numexp}, we compare the methods for two instances of the IEP. Section~\ref{sec:application} makes the connection to rational approximation by expressing the lightning and AAA approximants in our framework. Finally, Section~\ref{sec:conclusion} provides concluding remarks and gives some future research directions. Both algorithm implementations and all experiments are publicly available at
    \begin{center}
        \url{https://gitlab.kuleuven.be/numa/public/iep-ratvec}.
    \end{center}

\section{Problem formulation}\label{sec:problemformulation}
In this section, we discuss the necessary machinery needed to reformulate the problem of constructing a set of orthonormal vector-valued rational functions in terms of an inverse eigenvalue problem. These vectors are orthonormal with respect to a certain given discrete inner product. For the sake of exposition, this paper deals with the case of $2$-dimensional vectors. A generalization to more than two components is straightforward. The vectors are elements of the following complex vector space.
\begin{definition}\label{def:rationalvectors}
\rev{
    Let $\zeta_1, \zeta_2, \ldots, \zeta_{m-2} \in \C$ be pairwise distinct poles with tangential directions $\vec d_1, \vec d_2, \ldots, \vec d_{m-2} \in \C^2\setminus\{\vec 0\}$, and collect them in the matrices
    \begin{equation*}
        \Lambda_m = \operatorname{diag}(\zeta_1, \ldots, \zeta_{m-2}) \in \C^{(m-2)\times(m-2)}, \qquad D_m = \begin{bmatrix}
            \vec d_1, \ldots, \vec d_{m-2}
        \end{bmatrix} \in \C^{2\times(m-2)}.
    \end{equation*}
    The space of vector-valued rational functions is defined as $\mathcal{R}_1=\operatorname{span}\{\vec e_1\}$ and for $m\ge2$,
    \begin{equation*}
        \mathcal{R}_m = \left\{ \vec \phi(z) = \begin{bmatrix} \beta_1 \\ \beta_2 \end{bmatrix} + D_m(zI-\Lambda_m)^{-1}\vec \alpha \;:\; \beta_1, \beta_2 \in \C,\; \vec \alpha \in \C^{m-2} \right\} \subset \C(z)^2.
    \end{equation*}}
\end{definition}
\rev{The matrices $\Lambda_m$ and $D_m$ fix the poles and the tangential directions, while $\beta_1$, $\beta_2$ and $\vec \alpha$, with entries $\alpha_k$, are free parameters. This is the realization (state-space) form of a rational function from systems theory, with state matrix $\Lambda_m$ and output matrix $D_m$, which is also used in tangential interpolation~\cite{BallGohbergRodman1990,MayoAntoulas2007}. Since $\Lambda_m$ is diagonal, every $\vec \phi \in \mathcal{R}_m$ is of the form}
\begin{equation*}
    \rev{\vec \phi(z) =\begin{bmatrix}
    \phi_1(z) \\ \phi_2(z)
    \end{bmatrix}= \begin{bmatrix} \beta_1 \\ \beta_2 \end{bmatrix} + \sum_{k=1}^{m-2} \frac{\alpha_k \vec d_k}{z-\zeta_k}.}
\end{equation*}
\rev{Hence, the residue of $\vec \phi$ at $\zeta_k$ is $\alpha_k \vec d_k$: its direction is fixed by $\vec d_k$, and only the scalar factor $\alpha_k$ is free. Choosing $(\beta_1, \beta_2) \in \{(1,0), (0,1)\}$ and $\vec \alpha = \vec 0$, or $\beta_1 = \beta_2 = 0$ and $\vec \alpha$ equal to the $k$th unit vector, gives the basis $\{\vec r_1, \ldots, \vec r_m\}$ of $\mathcal{R}_m$, where}
\begin{equation*}
    \rev{\vec r_1 = \vec e_1, \qquad \vec r_2 = \vec e_2, \qquad \vec r_{k+2}(z) = \frac{\vec d_k}{z-\zeta_k}, \quad k = 1, \ldots, m-2.}
\end{equation*}
\rev{The spaces are nested, $\mathcal{R}_1 \subset \mathcal{R}_2 \subset \cdots \subset \mathcal{R}_m$, and $\dim \mathcal{R}_k = k$ for all $1\le k \le m$. Going from $\mathcal{R}_{k+1}$ to $\mathcal{R}_{k+2}$ adds the basis vector $\vec r_{k+2}$, i.e., the pole $\zeta_k$ is added with tangential direction $\vec d_k$, i.e., $\Lambda_{k+2}$ and $D_{k+2}$ are obtained by appending $\zeta_k$ and $\vec d_k$ to $\Lambda_{k+1}$ and $D_{k+1}$.} Repeated poles can be included by using higher powers of $(z-\zeta_k)^{-1}$\rev{, i.e., Jordan blocks in $\Lambda_m$, and poles at infinity by using polynomial terms; for ease of exposition, we do not consider these cases, except for poles at infinity in Remark~\ref{rem:infpoles}.}
\begin{example}
\rev{
    For $m=5$, consider pairwise distinct poles $\zeta_1, \zeta_2, \zeta_3 \in \C$ with tangential directions $\vec d_1 = \vec e_1$, $\vec d_2 = \vec e_2$ and $\vec d_3 = [1, 1]^T$, i.e.,
    \begin{equation*}
        \Lambda_5 = \begin{bmatrix}
            \zeta_1 & & \\
            & \zeta_2 & \\
            & & \zeta_3
        \end{bmatrix}, \qquad D_5 = \begin{bmatrix}
            1 & 0 & 1\\
            0 & 1 & 1
        \end{bmatrix}.
    \end{equation*}
    Every $\vec \phi \in \mathcal{R}_5$ is then of the form
    \begin{equation*}
        \vec \phi(z) = \begin{bmatrix} \beta_1 \\ \beta_2 \end{bmatrix} + D_5(zI-\Lambda_5)^{-1}\vec \alpha = \begin{bmatrix}
            \beta_1 + \dfrac{\alpha_1}{z-\zeta_1} + \dfrac{\alpha_3}{z-\zeta_3}\\[2ex]
            \beta_2 + \dfrac{\alpha_2}{z-\zeta_2} + \dfrac{\alpha_3}{z-\zeta_3}
        \end{bmatrix}.
    \end{equation*}
    The pole $\zeta_1$ only appears in the first component function and $\zeta_2$ only in the second one, while $\zeta_3$ appears in both, with equal residues $\alpha_3$.}
\end{example}

To define the notion of orthogonality, we introduce a discrete inner product.
\begin{definition}\label{def:innerprod}
    Let $z_{i}\in \C,\;i=1,\ldots,n$ be a set of nodes, all distinct from the poles, i.e., $z_{i} \neq \zeta_{k}$ for all $i,k$. For each node $z_{i}$, we have a corresponding weight vector $\bm{w}_{i} \in \C^2$, different from the zero vector. Let $\bm{\phi}(z), \bm{\psi}(z) \in \mathcal{R}_n$, then define the (pseudo)-inner product as
    \begin{equation}\label{eq:innerprod}
        \left\langle \bm{\phi}(z), \bm{\psi}(z) \right\rangle_{\mathcal{R}_n} = \sum_{i=1}^{n} \vec \psi(z_{i})^H\vec w_i\vec w_{i}^H\vec \phi(z_{i}), 
    \end{equation}
    where ${\cdot}^H$ denotes the complex conjugate transpose.
\end{definition}
Note that for general weights $\vec w_i$, equation \eqref{eq:innerprod} does not define a true inner product. To have an inner product, we need that $\langle \cdot,\cdot \rangle_{\mathcal{R}_n}$ is positive-definite, i.e., $\langle \vec \phi(z), \vec \phi(z) \rangle_{\mathcal{R}_n} = 0$ if and only if $\vec \phi(z) = \vec 0$. A necessary but insufficient condition is that the weight matrix \begin{equation*}
    W=\begin{bmatrix}
         \vec w_{1},\vec w_{2}, \ldots, \vec w_{n}
\end{bmatrix}^H \in \C^{n\times 2}
\end{equation*} has full column rank. Then, $W^HW$ is a matrix corresponding to a positive-definite inner product on $\C^{2}$.
To ensure positive-definiteness on $\mathcal{R}_n$, we assume that $\vec w_{i}^H\vec \phi(z_{i})$ does not vanish for all $i$\rev{; in particular, none of the weights $\vec w_{i}$ can be zero}. If this assumption fails, it results in an algorithmic breakdown when solving the following problem.


\begin{problem}\label{problem:general}
    Let $\langle \cdot, \cdot \rangle_{\mathcal{R}_n}$ be a discrete inner product as in Definition  \ref{def:innerprod}, construct an orthonormal set $\left \{ \vec \phi_i \right \}_{i=1}^{n}$, where $\vec \phi_{i} \in \mathcal{R}_i \setminus \mathcal{R}_{i-1}$, meaning that
    $\left \langle \vec \phi_i,\vec \phi_j \right \rangle_{\mathcal{R}_n} = \delta_{i,j}$.
\end{problem}
We now show how to cast Problem~\ref{problem:general} as an IEP. Consider the ``economical''\footnote{Only computes the essential part of the QR factorization, i.e. $Q_W\in \C^{n{\times} 2}$ and $R_W \in \C^{2\times 2}$.} QR factorization of the weight matrix $W$:
\begin{equation*}
    W = Q_WR_W.
\end{equation*}
We then choose $\begin{bmatrix}
    \vec \phi_1 & \vec \phi_2
\end{bmatrix} = R^{-1}_W = \begin{bmatrix}
    \vec r_{W,1} & \vec r_{W,2}
\end{bmatrix}$, where $\vec r_{W,i}$ denotes the $i$th column of $R^{-1}_W$. Since $W$ has full column rank, $R_W$ is nonsingular, so $\vec \phi_1$ and $\vec \phi_2$ span $\mathcal{R}_2$. They are also orthonormal: since $W\vec r_{W,i} = Q_W R_W R_W^{-1}\vec e_i = Q_W\vec e_i$, we have
\begin{align*}
    \begin{split}
        \|\vec \phi_i\|_{\mathcal{R}_n}^2 &= \langle \vec \phi_i, \vec \phi_i\rangle_{\mathcal{R}_n} = \langle W\vec r_{W,i}, W\vec r_{W,i}\rangle = \|Q_W\vec e_i\|^2 = 1, \;\; i\in\{1,2\} \\
        \langle\vec \phi_1,\vec \phi_2 \rangle_{\mathcal{R}_n} &= \langle W\vec r_{W,1}, W\vec{r}_{W,2} \rangle = \langle Q_W\vec e_1, Q_W\vec e_2\rangle = 0, 
    \end{split}
\end{align*}
where $\langle \cdot, \cdot \rangle$ denotes the standard Euclidean inner product. Let $K,H \in \C^{\infty\times \infty}$ be two infinite 2-Hessenberg matrices, i.e., the elements $k_{ij} = h_{ij} =0,\; \text{for}\; i > j+2$, and consider the recurrence relation:
\begin{equation}\label{eq:reccrelation1}
    z\begin{bmatrix}
        \vec \phi_1(z), \vec \phi_2(z),\ldots
    \end{bmatrix} K = \begin{bmatrix}
        \vec \phi_1(z),  \vec \phi_2(z),\ldots
    \end{bmatrix} H. 
\end{equation}
Looking at the $i$th column of this relation gives
\begin{equation*}
    \left(zk_{i+2,i} - h_{i+2,i} \right) \vec \phi_{i+2}(z) = -\sum_{j=1}^{i+1}\left( zk_{ji}- h_{ji}\right) \vec \phi_j(z) =: \begin{bmatrix}
        \varphi_{i,1}(z) \\
        \varphi_{i,2}(z)
    \end{bmatrix}.
\end{equation*}
Thus, the new vector $\vec \phi_{i+2}$, which expands $\mathcal{R}_{i+1}$ to $\mathcal{R}_{i+2}$, is a combination of the previous vectors $\vec \phi_j$ and $z\vec \phi_j$, $j=1,\ldots,i+1$, divided by the linear term $zk_{i+2,i}-h_{i+2,i}$. \rev{The component functions $\varphi_{i,1}(z)$ and $\varphi_{i,2}(z)$ only contain poles already present in $\mathcal{R}_{i+1}$, and grow at most linearly as $z\to\infty$. We impose $\zeta_i = h_{i+2,i}/k_{i+2,i}$. Dividing by $k_{i+2,i}(z-\zeta_i)$ then removes this linear growth and introduces the simple pole $\zeta_i$ in both component functions. The residue of $\vec \phi_{i+2}$ at $\zeta_i$ equals $[\varphi_{i,1}(\zeta_i), \varphi_{i,2}(\zeta_i)]^T/k_{i+2,i}$, which depends on the recurrence coefficients, whereas Definition~\ref{def:rationalvectors} requires it to be parallel to $\vec d_i$. Write $\vec d_i = [a_i, b_i]^T$ and $\vec d_i^{\perp} = [b_i, -a_i]^T$.} Lemma~\ref{lemma:poles} solves this problem by showing that the recurrence coefficients can be chosen such that 
\rev{$$(\vec d_i^{\perp})^T [\varphi_{i,1}(\zeta_i), \varphi_{i,2}(\zeta_i)]^T = 0,$$ which makes the residue at $\zeta_i$ parallel to $\vec d_i$}. \rev{As a result, $\vec \phi_{i+2} \in \mathcal{R}_{i+1} + \operatorname{span}\{\vec r_{i+2}\} = \mathcal{R}_{i+2}$. Because $R_W^{-1}$ is upper triangular, we have $\vec \phi_1 \in \mathcal{R}_1$ and $\vec \phi_2 \in \mathcal{R}_2$, and induction shows that $\vec \phi_i \in \mathcal{R}_i$ for all $i$. Moreover, $\vec \phi_{i} \notin \mathcal{R}_{i-1}$, because it has unit norm and is orthogonal to $\vec \phi_1,\ldots,\vec \phi_{i-1}$, which span $\mathcal{R}_{i-1}$. Hence, we have $\vec \phi_i \in \mathcal{R}_i \setminus \mathcal{R}_{i-1}$ for all $i$.}

The orthonormality of the vectors $\vec \phi_j$ can be expressed in terms of the matrix $Q\in \C^{n\times\infty}$ with entries
\begin{equation*}
    Q_{ij} = \vec w^H_i\vec \phi_j(z_i), \; i=1,\ldots,n,\; j=1,2,\ldots.
\end{equation*}
\rev{Denoting the $j$th column of $Q$ by $\vec q_j \in \C^n$, Definition~\ref{def:innerprod} gives $\langle \vec \phi_j, \vec \phi_k \rangle_{\mathcal{R}_n} = \vec q_k^H \vec q_j$. Hence, the orthonormality of all $\vec \phi_j$'s is written as
\begin{equation*}
    Q^HQ = \operatorname{diag}(1,\ldots,1,\|\vec q_{n+1}\|^2,\|\vec q_{n+2}\|^2,\ldots).
\end{equation*}
Consequently, the matrix $Q_n$ containing the first $n$ columns of $Q$ is unitary. Its columns form an orthonormal basis of $\C^n$, so every other column $\vec q_j$, $j>n$, is orthogonal to this basis and therefore zero, i.e., $\vec w_i^H\vec \phi_{j}(z_i) = 0$ for all $i$ and $j>n$.} Taking the first $n$ columns of equation~\eqref{eq:reccrelation1}, evaluating in all nodes $z_i$, and multiplying with $\vec w_i^H$ results in
\begin{equation*}
    ZQ_nK_n=Q_nH_n,
\end{equation*}
where $K_n$ and $H_n$ are the leading principal $n\times n$ submatrices of $K$ and $H$.
We summarize all the above conditions in the following IEP.
\begin{problem}{(\textbf{IEP})}\label{problem:IEP}
    Given a diagonal matrix $Z \in \C^{n\times n}$ of nodes of the form $\operatorname{diag}\left([
        z_{i}]_{i=1}^{n}\right)$, weight matrix $W=\begin{bmatrix}
        \vec w_{1},\vec w_{2}, \ldots, \vec w_{n}
    \end{bmatrix}^H \in \C^{n\times 2}$ and poles $\zeta_1,\ldots,\zeta_{n-2} \in \C$ \rev{with tangential directions $\vec d_1,\ldots,\vec d_{n-2}$ as in Definition~\ref{def:rationalvectors}}. Construct a $2$-Hessenberg pencil $\left( H_n, K_n\right)$, and a unitary matrix $Q_n$ both belonging to $\C^{n \times n}$, such that
    \begin{enumerate}[label=(D\arabic*)]
        \item\label{cond:1} $Z Q_n K_n = Q_n H_n$,
        \item\label{cond:2} $W = Q_n\begin{bmatrix}
            \vec e_1& \vec e_2
        \end{bmatrix} R$ (``economic'' QR factorization),
        \item\label{cond:3} \rev{$\zeta_i = \frac{h_{i+2,i}}{k_{i+2,i}}, i=1,\ldots,n-2$,}
        \item\label{cond:4} $\vec \phi_i \in \mathcal{R}_i \setminus \mathcal{R}_{i-1}$, $i=1,\ldots,n$, with $\vec \phi_i$ generated by~\eqref{eq:reccrelation1}.
    \end{enumerate}
\end{problem}
 The final condition~\ref{cond:4} is essential because our chosen configuration for $\mathcal{R}_i$ introduces \rev{each pole with a prescribed tangential direction}. As previously discussed, the matrices $K_n$ and $H_n$ must be selected carefully, to ensure that the pole \rev{has the correct tangential direction}.

As we will see in Section \ref{sec:krylov}, a Stieltjes procedure on $\mathcal{R}_n$ leads to the same conditions as described in the above IEP. Furthermore, it is straightforward to see how the IEP generalizes to vectors of arbitrary length $k$. Indeed, the matrices $H_n$ and $K_n$ are $k$-Hessenberg and the poles equal the ratios of the $k$-th lower diagonals. From now on, we often omit the subscript $n$ of $K_n$, $H_n$ and $Q_n$ when the size is clear from the context.

\section{Updating Algorithm}\label{sec:updating}
This section presents an updating procedure for recursively solving Problem \ref{problem:IEP}. The idea is to start from a solution for $\mathcal{R}_n$ and update it to obtain the solution for $\mathcal{R}_{n+1} \supset \mathcal{R}_n$. The base cases $\mathcal{R}_1$ and $\mathcal{R}_2$ are given in the following lemma.
\begin{lemma}
    There exists a solution to Problem \ref{problem:IEP} for $n=1$ and $n=2$. 
\end{lemma}
\begin{proof}
    The case $n=1$ is trivial since we can set $Z = z_1$, $(H,K) = (z_1,1)$ and $Q=1$.
    For the case $n=2$, we have to find $2{\times}2$-matrices $H,K$ and a unitary $Q$ such that conditions \ref{cond:1} and \ref{cond:2} are satisfied. Note that the third and fourth conditions are trivially satisfied. Let $\widetilde{Q}=\widetilde{K}=I_2$ and $\widetilde{H}=Z=\operatorname{diag}([z_1\; z_2])$, then $Z\widetilde{Q}\widetilde{K} = \widetilde{Q}\widetilde{H}$ is trivially true. Now, consider a $QR$-decomposition of the weight matrix 
    \begin{equation*}
    W_2 = \begin{bmatrix}
        \vec w^H_{1} \\
        \vec w^H_{2}
    \end{bmatrix} = {Q_1}\begin{bmatrix}
        r_{11} & r_{12} \\
        0 & r_{22}
    \end{bmatrix}.    
    \end{equation*}
    It is easy to check that $(H,K) = (Q_1^H\widetilde{H}, Q_1^H\widetilde{K})$ and $Q = Q_1$ satisfy the requirements.
\end{proof}
The induction step in the updating procedure uses a special type of unitary matrices, namely rotations, see, e.g., the monograph by Golub and Van Loan~\cite[Chap. 5]{GolubVanLoan2013}. These are essentially $2{\times2}$ unitary matrices of the form
\begin{equation}\label{eq:Givens}
    G_{i,j} = 
       \begin{bNiceMatrix}[last-row,last-col,nullify-dots]
1& & \Vdots & & & & \Vdots \\
& \Ddots[line-style=standard] \\
\Cdots & & \bar{c}& \Cdots & & &  -\bar{s} &  & \Cdots & \leftarrow i \\
& & \Vdots&1  \\
& & & &  \Ddots[line-style=standard] & & \Vdots \\
& & & & & 1 \\
\Cdots & & s& \Cdots & & \Cdots & c & &  \Cdots & \leftarrow j \\
& & & & & & & \Ddots[line-style=standard] \\
& & \Vdots&  & & &  \Vdots &  &1 \\
& &   \overset{\uparrow}{i} & & & & \overset{\uparrow}{j} \\
\end{bNiceMatrix},
\end{equation}
where $s,c\in \C$ and $\bar{c}c+\bar{s}s=1$. Multiplying a matrix from the left by $G_{i,j}$ only changes its rows $i$ and $j$, whereas multiplying from the right only changes columns $i$ and $j$. In our updating algorithm, we mostly use these rotations to eliminate some element in a matrix or vector. For example, one can find an appropriate $G_{1,2} \in \C^{2\times2}$ such that $$G_{1,2}[r_1\ \ r_2]^T = [r\ \ 0]^T.$$

Before showing how to recursively solve Problem~\ref{problem:IEP}, we need two auxiliary results. The first lemma shows how to transform lower triangular ${2{\times} 2}$ pencils to upper triangular pencils by using rotations, without changing the location of the eigenvalues, because they correspond to the poles in the pencil $(H,K)$, as is shown in the proof of Theorem~\ref{thm:solveIEP}.
\begin{lemma}\label{lemma:swapzeros}
    Let $(A,B)$ be a pencil of lower triangular $2{\times}2$-matrices with generalized eigenvalues $\lambda_1 = a_{11}/b_{11}$ and $\lambda_2 = a_{22}/b_{22}$. \rev{Then there exist rotations $G_L, G_R \in \C^{2\times2}$ such that $G_L(A,B)G_R = (\widehat{A}, \widehat{B})$ is a pencil of upper triangular $2{\times}2$-matrices whose eigenvalues remain in the same position on the diagonal, i.e.,
    \begin{equation*}
        \widehat{a}_{11}/\widehat{b}_{11} = \lambda_1 \quad\text{and}\quad \widehat{a}_{22}/\widehat{b}_{22} = \lambda_2.
    \end{equation*}}
\end{lemma}
\begin{proof}
    \rev{Pick $G_L = \begin{bmatrix}
        0 & 1 \\
        -1 & 0
    \end{bmatrix}$ and $G_R=G_L^{-1} = \begin{bmatrix}
        0 & -1 \\
        1 & 0
    \end{bmatrix}$. The pencil $G_L(A,B)G_R$ now has the eigenvalues in flipped positions, thus we can swap them with the technique described by Van Dooren~\cite{VanDooren1981}, which provides two additional rotations $\widetilde G_L, \widetilde G_R \in \C^{2\times 2}$ such that $\widetilde G_L G_L(A,B)G_R\widetilde G_R$ is upper triangular, with the eigenvalues of $(A,B)$ back in positions $(1,1)$ and $(2,2)$, respectively.\footnote{This can be done with the \texttt{ordqz} routine in \texttt{MATLAB}.} The result is proven by noting that the composition of two rotations is a rotation.}
    
\end{proof}

 The following result gives a constructive proof for the introduction of a new pole $\zeta_{n-2}$ in a pencil $(H,K)$ such that $\vec \phi_n \in \mathcal{R}_n \setminus \mathcal{R}_{n-1}$.
\begin{lemma}\label{lemma:poles}
Assume we are in the same setting as Problem~\ref{problem:IEP}. Let $(H,K)$, with $H,K \in \mathbb{C}^{n\times n}$, be a $2$-Hessenberg pencil that satisfies \rev{all conditions of Problem~\ref{problem:IEP}, except that the final pole $\zeta_{n-2}$ is not yet introduced with the correct tangential direction $\vec d_{n-2}$, i.e., except for
\begin{equation}\label{eq:lastpoleconds}
    \zeta_{n-2} = \frac{h_{n,n-2}}{k_{n,n-2}} \;\;\text{\ref{cond:3}} \qquad \text{and} \qquad \vec \phi_n \in \mathcal{R}_n \setminus \mathcal{R}_{n-1} \;\;\text{\ref{cond:4}}.
\end{equation}}
Then there exists a matrix $G_R \in \mathbb{C}^{n \times n}$, which can be expressed as a product of rotations, such that the transformed pencil $(\widehat{H}, \widehat{K}) = (H,K)G_R$ fulfills all conditions of Problem~\ref{problem:IEP}, including~\eqref{eq:lastpoleconds}.
\end{lemma}
\begin{proof}
The two conditions boil down to adding the pole \rev{with tangential direction $\vec d_{n-2}$}.
As will become clear later, we need the pole on both the first and second subdiagonal of $(H,K)$
\begin{equation}\label{eq:poleplacement}
    \zeta_{n-2} = \frac{h_{n,n-1}}{k_{n,n-1}} = \frac{h_{n,n-2}}{k_{n,n-2}}.
\end{equation}
This can be done with two rotations $G_{n-1,n}$ and $G_{n-2,n}$ acting on the columns such that $(\widetilde{H}, \widetilde{K}) = (H,K)G_{n-1,n}G_{n-2,n}$ has the ratios equal to the pole $\zeta_{n-2}$. For example, to enforce the pole $\zeta_{n-2} = \nu/\mu$ on the first subdiagonal, we need that
\begin{equation*}
    \zeta_{n-2}=\frac{\nu}{\mu} = \frac{\tilde{h}_{n,n-1}\overline{c} - s\tilde{h}_{n,n}}{\tilde{k}_{n,n-1}\overline{c} - s\tilde{k}_{n,n}},
\end{equation*}
where $c,s \in \C$ are the elements of $G_{n-1,n}$, cf.~\eqref{eq:Givens}.
Some calculations show that we can take $G_{n-1,n}$ such that
\begin{equation*}
    G_{n-1,n}\begin{bmatrix}
        \mu\tilde{h}_{n,n-1}-\nu\tilde{k}_{n,n-1} \\
        -\mu\tilde{h}_{n,n}+\nu\tilde{k}_{n,n}
    \end{bmatrix} = \begin{bmatrix}
        \times \\
        0
    \end{bmatrix}.
\end{equation*}
The $(n-2)$th column of $(\widetilde{H}, \widetilde{K})$ yields the recursion relation
\begin{align*}
\begin{split}
    \left( z\tilde{k}_{n,n-2} - \tilde{h}_{n,n-2}\right)\vec \phi_{n}(z) &= -\sum_{i=1}^{n-1}\left( z\tilde{k}_{i,n-2}- \tilde{h}_{i,n-2}\right) \vec \phi_i(z) \\
    &= \begin{bmatrix}
        \tilde{\varphi}_{n-2,1}(z) \\
        \tilde{\varphi}_{n-2,2}(z)
    \end{bmatrix}.
    \end{split}
\end{align*}
Note that the pole $\zeta_{n-2} = \tilde{h}_{n,n-2}/\tilde{k}_{n,n-2}$ gets added in both component functions of $\vec \phi_n(z)$. However, \rev{its residue at $\zeta_{n-2}$ is proportional to $[\tilde{\varphi}_{n-2,1}(\zeta_{n-2}), \tilde{\varphi}_{n-2,2}(\zeta_{n-2})]^T$, which is in general not parallel to $\vec d_{n-2}$. It becomes parallel by choosing the recursion coefficients such that
\begin{equation*}
    (\vec d_{n-2}^{\perp})^T[\tilde{\varphi}_{n-2,1}(\zeta_{n-2}), \tilde{\varphi}_{n-2,2}(\zeta_{n-2})]^T = 0.
\end{equation*}}
This can be done by using a third and final rotation $G_{n-2,n-1}$ acting on the columns such that the pencil $(\widehat{H},\widehat{K}) = (\widetilde{H},\widetilde{K})G_{n-2,n-1}$ solves Problem~\ref{problem:IEP}.
Consider the vector-valued rational function
\begin{equation}\label{eq:rationalvector}
    \begin{bmatrix}
        \tilde{\varphi}_{n-1,1}(z) \\
        \tilde{\varphi}_{n-1,2}(z)
    \end{bmatrix} = -\sum_{i=1}^{n-1}\left( z\tilde{k}_{i,n-1}- \tilde{h}_{i,n-1}\right) \vec \phi_i(z).
\end{equation}
\rev{Let $\tilde{\psi}_j(z) = (\vec d_{n-2}^{\perp})^T[\tilde{\varphi}_{j,1}(z), \tilde{\varphi}_{j,2}(z)]^T$ for $j = n-2, n-1$.} Choosing $G_{n-2,n-1}$ such that
\begin{equation}\label{eq:finalrotation}
    \begin{bmatrix}
        \rev{\tilde{\psi}_{n-2}(\zeta_{n-2})}&  \rev{\tilde{\psi}_{n-1}(\zeta_{n-2})}
     \end{bmatrix} G_{n-2,n-1} = \begin{bmatrix}
         0 & \times
     \end{bmatrix}
\end{equation}
takes a linear combination of \rev{$\tilde{\psi}_{n-2}$ and $\tilde{\psi}_{n-1}$} such that the new function \rev{$\hat{\psi}_{n-2}(z)$} has a zero at $\zeta_{n-2}$. Note that, after this final rotation, we obtain $\zeta_{n-2} = \hat{h}_{n,n-2}/\hat{k}_{n,n-2}$, precisely because the pole was placed in the first and second subdiagonals during the previous steps, cf.~\eqref{eq:poleplacement}.
\end{proof}
In our updating method, we apply the above lemma for each pole $\zeta_k$. We need the evaluation of rational functions $\varphi_{i,j}(z)$ in the pole $\zeta_k$ as illustrated by equations~\eqref{eq:rationalvector} and~\eqref{eq:finalrotation}. This evaluation can be performed by applying the recurrence relation, requiring $\mathcal{O}(n^2)$ operations. More concretely, evaluation of $\vec \phi_j(z)$ in $z \in \C$ is done by solving the following linear system
\begin{equation}\label{eq:linearsystem}
    \begin{bmatrix}
        \vec \phi_1(z) & \vec \phi_2(z) & \ldots & \vec \phi_n(z) 
    \end{bmatrix}
    \begin{bNiceArray}{ccccc}[margin=3pt]
          \Block[borders={bottom,right}]{1-1}{ I_2}& \Block[borders={left}]{3-4}{ z\widehat{K}-\widehat{H}}& & &\\
           & &  & &\\
           && & &
        \end{bNiceArray} = \begin{bmatrix}
            \vec \phi_1(z) & \vec \phi_2(z) & 0 \ldots 0
        \end{bmatrix},
\end{equation}
where $\widehat{K},\widehat{H} \in \C^{n \times n-2}$ are equal to $K$ and $H$ but with the two last columns removed. \rev{The matrix $z\widehat{K}-\widehat{H}$ is the pencil $(\widehat{H},\widehat{K})$ evaluated at $z$.} Using the lemmas, we now give a constructive proof of solving Problem~\ref{problem:IEP} inductively.   
\begin{theorem}\label{thm:solveIEP}    
Let $Z_k \in \C^{k\times k}$, weight matrix $W_k \in \C^{k\times 2}$, and poles $\zeta_1,\ldots,\zeta_{k-2}$ \rev{with tangential directions $\vec d_1,\ldots,\vec d_{k-2}$} define an IEP as given in Problem \ref{problem:IEP} and let $(H,K)$ and $Q$ be the solution. Assume we add an additional node $z_{k+1}$, weight $\vec w_{k+1}$, and pole $\zeta_{k-1}$ \rev{with tangential direction $\vec d_{k-1}$} to define an IEP of dimension $k+1$. Then, this IEP has the solution $(\widehat{H}, \widehat{K})$ and $\widehat{Q}$ which can be written as:
    \begin{align*}
        (\widehat{H}, \widehat{K}) &= Q_L\left(\begin{bmatrix}
        H&  \\
        &z_{k+1}
        \end{bmatrix},\begin{bmatrix}
        K&  \\
        &1
        \end{bmatrix}\right)Q_R \\
        \widehat{Q} &= Q_L \begin{bmatrix}
            Q & \\
            & 1
        \end{bmatrix},
    \end{align*}
    where $Q_L,Q_R$ are matrices composed of rotations.
\end{theorem}
\begin{proof}
We give a proof by construction.
   First, a new node is added in the following way
    \begin{equation*}
        \begin{bNiceArray}{ccc}[margin]
          \Block[borders={bottom,right}]{2-2}{Z}& & \\
           & & \\
           && z_{k+1}
        \end{bNiceArray}\begin{bNiceArray}{ccc}[margin]
          \Block[borders={bottom,right}]{2-2}{Q}& & \\
           & & \\
           && 1
        \end{bNiceArray}\begin{bNiceArray}{ccc}[margin]
          \Block[borders={bottom,right}]{2-2}{K}& & \\
           & & \\
           && 1
        \end{bNiceArray}=\begin{bNiceArray}{ccc}[margin]
          \Block[borders={bottom,right}]{2-2}{Q}& & \\
           & & \\
           && 1
        \end{bNiceArray}\begin{bNiceArray}{ccc}[margin]
          \Block[borders={bottom,right}]{2-2}{H}& & \\
           & & \\
           && z_{k+1}
        \end{bNiceArray}.
    \end{equation*}
    Condition~\ref{cond:2} for the IEP of size $k$ gives
    \begin{equation*}
         \begin{bNiceArray}{ccc}[margin]
          \Block[borders={bottom,right}]{2-2}{Q^H}& & \\
           & & \\
           && 1
        \end{bNiceArray}W_{k+1} =
        \begin{bNiceArray}{cc}[margin]
            \times & \times \\
            0 & \times \\
            0 & 0 \\
            \vdots & \vdots \\
            \hline 
            \Block{1-2}{\rule{0pt}{10pt}\,\vec w^H_{k+1}}
        \end{bNiceArray}.
    \end{equation*}
    Now apply two rotations $G_{1,k+1}$, $G_{2,k+1}$ such that the last row is eliminated,
\begin{equation*}
    \underbrace{
        G_{2,k+1}G_{1,k+1}
        \begin{bNiceArray}{ccc}[margin]
            \Block[borders={bottom,right}]{2-2}{Q^H} & & \\
             & & \\
             & & 1
        \end{bNiceArray}
    }_{\mathclap{\coloneqq \widetilde{Q}^H}} 
    W_{k+1} = 
    \begin{bNiceArray}{cc}[margin]
        \times & \times \\
        0 & \times \\
        0 & 0 \\
        \vdots & \vdots \\
        0 & 0
    \end{bNiceArray}
\end{equation*}
Note that $\widetilde{Q}$ makes condition \ref{cond:2} automatically fulfilled for the IEP of size $k+1$. Furthermore, by defining
\begin{equation*}
    (\widetilde{H}, \widetilde{K}) = G^H_{1,k+1}G^H_{2,k+1}\left(\begin{bNiceArray}{ccc}[margin]
          \Block[borders={bottom,right}]{2-2}{H}& & \\
           & & \\
           && z_{k+1}
        \end{bNiceArray}, \begin{bNiceArray}{ccc}[margin]
          \Block[borders={bottom,right}]{2-2}{K}& & \\
           & & \\
           && 1
        \end{bNiceArray} \right),
\end{equation*}
the relation $Z_{k+1}\widetilde{Q}\widetilde{K} = \widetilde{Q}\widetilde{H}$ also remains valid. However, the pencil $(\widetilde{H}, \widetilde{K})$ does not yet have the required 2-Hessenberg structure with pole $\zeta_{k-1}$. For example, for $k=5$, by looking at the rows on which the rotations $G_{1,k+1}^H$ and $G_{2,k+1}^H$ act, we obtain a pencil having the following structure 
\begin{equation*}
\left(\widetilde{H}, \widetilde{K}\right) =
    \begin{bNiceMatrix}[nullify-dots,margin]
    \CodeBefore
    \cellcolor{red!15}{6-1,3-1,3-6,6-6}
    \Body
        \times & \times & \times & \times & \times & \times\\
        \times & \times & \times & \times & \times & \times\\
        \times &\times & \times & \times &   \times & \\
          & \times & \times & \times & \times &   \\
          & & \times & \times & \times &  \\
          \color{red}{\times} & \color{red}{\times} & \color{red}{\times} & \times & \times & \times\\
    \end{bNiceMatrix}, \begin{bNiceMatrix}[nullify-dots,margin]
    \CodeBefore
    \cellcolor{red!15}{6-1,3-1,3-6,6-6}
    \Body
        \times & \times & \times & \times & \times & \times\\
        \times & \times & \times & \times & \times & \times\\
        \times &\times & \times & \times &   \times & \\
          & \times & \times & \times & \times &   \\
          & & \times & \times & \times &  \\
          \color{red}{\times} & \color{red}{\times} & \color{red}{\times} & \times & \times & \times\\
    \end{bNiceMatrix}.
\end{equation*}
The pencil is 2-Hessenberg if the elements at positions $(6,1), (6,2)$ and $(6,3)$ can be removed.
This is possible by repeatedly using Lemma \ref{lemma:swapzeros}.
Consider the $2 {\times}2$ submatrices extracted from $(\widetilde{H}, \widetilde{K})$, formed by the intersections of the third and last rows with the first and last columns, as indicated by the highlighted entries. 
Inserting this pencil of $2{\times}2$-matrices in Lemma~\ref{lemma:swapzeros} gives two rotations $G_{L,1} = G_{3,6}$ and $G_{R,1} = G_{1,6}$ such that $G_{L,1}(\widetilde{H}, \widetilde{K})G_{R,1}$ has zeros in the lower left corners. Successive application of Lemma~\ref{lemma:swapzeros} for the unwanted elements at positions $(6,2)$ and $(6,3)$ generates the rotations $\{G_{L,2}, G_{R,2}\}$ and $\{G_{L,3}, G_{R,3}\}$, respectively.  In general, we apply Lemma \ref{lemma:swapzeros} $(k-2)$ times to obtain the desired pencil structure. These sequences of rotations $G_{L} = \prod_{i=1}^{k-2}G_{L,i}$ and $G_{R} = \prod_{i=1}^{k-2}G_{R,i}$, $i=1,\ldots,k-2$, are inserted, so that condition \ref{cond:1} remains valid,
\begin{equation*}
    Z_{k+1} 
    \underbrace{\left(\widetilde{Q}G^H_L\right)}_{\coloneqq \overline{Q}}
    \underbrace{G_L\widetilde{K}G_R}_{\coloneqq \overline{K}}
    = 
    \left(\widetilde{Q}G_L^H\right)
    \underbrace{G_L\widetilde{H}G_R}_{\coloneqq \overline{H}},
\end{equation*}
where $(\overline{H},\overline{K})$ is a 2-Hessenberg pencil. \rev{Condition~\ref{cond:2} remains satisfied, since multiplying $\widetilde{Q}$ from the right by $G_L^H$ leaves its first two columns unchanged.} Note that $(\overline{H},\overline{K})$ satisfies the conditions in Lemma~\ref{lemma:poles}.   
The final step is adding the pole $\zeta_{k-1} $\rev{with tangential direction $\vec d_{k-1}$}, which is exactly the procedure described in the proof of Lemma~\ref{lemma:poles}. This lemma gives a product of rotations $G_r$ such that the pencil $(\widehat{H},\widehat{K})=(\overline{H},\overline{K})G_r$ solves the IEP of dimension $k+1$.
\end{proof}

\begin{revblock}
\begin{example}\label{ex:updating}
We illustrate the steps in the proof of Theorem~\ref{thm:solveIEP} for $k=4$. Consider the nodes $z_i = i-1$, $i=1,\ldots,5$, the weights
\begin{equation*}
    \vec w_1 = \begin{bmatrix} 1\\ 1 \end{bmatrix},\quad
    \vec w_2 = \begin{bmatrix} 2\\ -1 \end{bmatrix},\quad
    \vec w_3 = \begin{bmatrix} 1\\ 2 \end{bmatrix},\quad
    \vec w_4 = \begin{bmatrix} -1\\ 1 \end{bmatrix},\quad
    \vec w_5 = \begin{bmatrix} 2\\ 1 \end{bmatrix},
\end{equation*}
and the poles $\zeta_1 = -1$, $\zeta_2 = -2$ and $\zeta_3 = 6$ with tangential directions $\vec d_1 = \vec e_1$, $\vec d_2 = \vec e_2$ and $\vec d_3 = \vec e_1$. All entries are rounded to two or three decimals.
Since $W_4^HW_4 = 7I_2$, the solution of the IEP of size $4$ has $R = \sqrt{7}\,I_2$ and
\begin{equation*}
    (H,K) = \left(
    \begin{bNiceMatrix}[r]
        0.88 & 0.34 & -1.03 & 0.70 \\
        -0.11 & -1.44 & -0.82 & -0.98 \\
        -0.38 & 0.33 & 1.08 & -0.94 \\
         & -0.44 & 0.34 & -2.29
    \end{bNiceMatrix},
    \begin{bNiceMatrix}[r]
        0.92 & 0.16 & -0.35 & 0.10 \\
        0.11 & -0.95 & -0.21 & -0.18 \\
        0.38 & -0.12 & 0.90 & -0.19 \\
         & 0.22 & -0.17 & -0.96
    \end{bNiceMatrix}\right),
\end{equation*}
with the poles $h_{3,1}/k_{3,1} = -1$ and $h_{4,2}/k_{4,2} = -2$ on the second subdiagonal.

\emph{Step 1: adding the node and weight.} We add the node $z_5 = 4$, the weight $\vec w_5$ and the pole $\zeta_3 = 6$. The rotations $G_{1,5}$ and $G_{2,5}$ eliminate the new weight row,
\begin{equation*}
    \begin{bmatrix} Q^H & \\ & 1 \end{bmatrix} W_5 =
    \begin{bmatrix} \sqrt{7} & 0 \\ 0 & \sqrt{7} \\ 0 & 0 \\ 0 & 0 \\ 2 & 1 \end{bmatrix}
    \xrightarrow{\;G_{1,5}\;}
    \begin{bmatrix} \sqrt{11} & 2/\sqrt{11} \\ 0 & \sqrt{7} \\ 0 & 0 \\ 0 & 0 \\ 0 & \sqrt{7/11} \end{bmatrix}
    \xrightarrow{\;G_{2,5}\;}
    \begin{bmatrix} \sqrt{11} & 2/\sqrt{11} \\ 0 & \sqrt{84/11} \\ 0 & 0 \\ 0 & 0 \\ 0 & 0 \end{bmatrix},
\end{equation*}
which gives $\widehat{R}$, the R-factor of $W_5$. Applying the same rotations to the rows of the extended pencil gives
\begin{equation*}
    (\widetilde{H}, \widetilde{K}) = \left(
    \begin{bNiceMatrix}[r,margin]
    \CodeBefore
    \cellcolor{red!15}{3-1,3-5,5-1,5-5}
    \Body
        0.70 & 0.27 & -0.82 & 0.56 & 2.41 \\
        -0.25 & -1.44 & -0.61 & -1.06 & 0.92 \\
        -0.38 & 0.33 & 1.08 & -0.94 & 0 \\
         & -0.44 & 0.34 & -2.29 & 0 \\
        \color{red}{-0.47} & \color{red}{0.22} & 0.83 & -0.12 & 3.06
    \end{bNiceMatrix},
    \begin{bNiceMatrix}[r,margin]
    \CodeBefore
    \cellcolor{red!15}{3-1,3-5,5-1,5-5}
    \Body
        0.73 & 0.13 & -0.28 & 0.08 & 0.60 \\
        -0.06 & -0.94 & -0.14 & -0.19 & 0.23 \\
        0.38 & -0.12 & 0.90 & -0.19 & 0 \\
         & 0.22 & -0.17 & -0.96 & 0 \\
        \color{red}{-0.56} & \color{red}{0.18} & 0.26 & 0.00 & 0.76
    \end{bNiceMatrix}\right).
\end{equation*}
The entries at positions $(5,1)$ and $(5,2)$, shown in red, violate the 2-Hessenberg structure.

\emph{Step 2: restoring the 2-Hessenberg structure.} The highlighted entries form a lower triangular $2{\times}2$ pencil. Lemma~\ref{lemma:swapzeros} gives a rotation $G_{L,1}$ acting on rows $3$ and $5$ and a rotation $G_{R,1}$ acting on columns $1$ and $5$ that eliminate the entries at position $(5,1)$, while keeping the ratio $-1$ at position $(3,1)$. Next, $G_{L,2}$ acting on rows $4$ and $5$ and $G_{R,2}$ acting on columns $2$ and $5$ eliminate the entries at position $(5,2)$, while keeping the ratio $-2$ at position $(4,2)$. The pencil $(\overline{H},\overline{K})$ is 2-Hessenberg, but $\bar h_{5,3}/\bar k_{5,3} = 1.375/0.772 \approx 1.78$ does not equal $\zeta_3$. Since $G_L^H$ only acts on columns $3$, $4$ and $5$ of $\widetilde{Q}$, the first two columns of $\overline{Q}$ and, hence, $\widehat{R}$ do not change.

\emph{Step 3: adding the pole.} Following the proof of Lemma~\ref{lemma:poles} with $n=5$, the rotations $G_{4,5}$ and $G_{3,5}$ acting on the columns place the pole on both subdiagonals: $h_{5,4}/k_{5,4} = h_{5,3}/k_{5,3} = 6$. At this point, $\vec\phi_5$ has a pole at $6$ in both components. Since $\vec d_3 = \vec e_1$, equation~\eqref{eq:finalrotation} only involves the second components of the vector-valued rational functions, which evaluate to $\tilde\varphi_{3,2}(6) \approx 1.824$ and $\tilde\varphi_{4,2}(6) \approx 2.716$. The rotation $G_{3,4}$ combines columns $3$ and $4$ so that $\hat\varphi_{3,2}(6) = 0$, which removes the pole from the second component of $\vec\phi_5$. The result is
\begin{equation*}
    (\widehat{H}, \widehat{K}) = \left(
    \begin{bNiceMatrix}[r]
        1.31 & -0.03 & 2.37 & 0.11 & -0.13 \\
        0.00 & -1.56 & 0.69 & -1.24 & -0.15 \\
        -0.50 & 0.38 & -2.04 & -0.30 & -0.36 \\
         & -0.44 & -0.80 & -2.03 & 0.35 \\
         &  & 1.56 & -1.01 & 2.21
    \end{bNiceMatrix},
    \begin{bNiceMatrix}[r]
        0.87 & 0.07 & 0.48 & -0.03 & -0.14 \\
        0.00 & -0.97 & 0.11 & -0.23 & -0.07 \\
        0.50 & -0.12 & -0.82 & 0.05 & 0.23 \\
         & 0.22 & -0.13 & -0.96 & -0.13 \\
         &  & 0.26 & -0.17 & 0.95
    \end{bNiceMatrix}\right),
\end{equation*}
with the poles $-1$, $-2$ and $6$ on the second subdiagonal. The corresponding unitary matrix $\widehat{Q}$, which we omit for brevity, has first two columns $W_5\widehat{R}^{-1}$, and the computed solution satisfies $\|Z_5\widehat{Q}\widehat{K} - \widehat{Q}\widehat{H}\|_2 \approx 10^{-15}$.
\end{example}
\end{revblock}


It is interesting to note that, at each step, $K_k$ is a unitary 2-Hessenberg matrix. The unitarity property allows for a data-sparse representation through a structured factorization of rotations. Mach, Van Barel, and Vandebril~\cite{MachVanBaVand2014} demonstrated, for the (non-vector) orthogonal rational functions, how to solve the IEP when the recurrence matrix is stored as a product of rotations. These rotations are manipulated—using operations such as \emph{turnover}, \emph{fusion}, and \emph{passing through} --- to maintain the factorized form throughout the algorithm. A similar approach could certainly be developed for vector-valued rational functions; however, in this paper, we restrict ourselves to the non-factorized form.

\section{Tangential rational Krylov subspaces}\label{sec:krylov}
This section shows how to solve Problem \ref{problem:general} with a Krylov-type algorithm by exploiting the connection with tangential rational Krylov subspaces. This can be seen as a vector generalization of the classical rational Arnoldi algorithm~\cite{Guttel2010}, \rev{and it is similar to the discretized Stieltjes procedure for computing orthogonal polynomials~\cite{Gautschi2004}}. Recall that for a matrix $A\in \C^{n\times n}$ and vector $v\in \C^n$, the standard Krylov space is defined as $\mathcal{K}_m(A,\vec v)=\operatorname{span}\{\vec v,A\vec v,\ldots,A^{m-1}\vec v\}$. A rational Krylov subspace can be seen as a standard Krylov space multiplied by some rational function evaluated in the matrix $A$.

\begin{definition}{(Rational Krylov Subspace)}
Consider a matrix $A\in \C^{n\times n}$, a vector $\vec v \in \C^n$, and pairwise distinct poles $\xi_1,\ldots,\xi_{m-1} \in \C$ disjoint from the spectrum of $A$. The rational Krylov subspace of order $m$ is defined as
\begin{equation*}
    \mathcal{Q}_m(A,\vec v) = \rev{\operatorname{range}\left[\vec v,\; (A-\xi_1I)^{-1}\vec v,\; \ldots,\; (A-\xi_{m-1}I)^{-1}\vec v\right].}
\end{equation*}
\end{definition}
By a partial fraction argument, $\mathcal{Q}_m(A,\vec v) = q_{m-1}(A)^{-1}\mathcal{K}_m(A,\vec v)$, with denominator polynomial $q_{m-1}(z) = \prod_{i=1}^{m-1} (z-\xi_i)$.
\rev{Tangential rational Krylov subspaces were introduced by Druskin, Simoncini and Zaslavsky~\cite{DruskinSimonciniZaslavsky2014} for the model reduction of multi-input, multi-output systems. They are spanned by the vectors $(A-s_iI)^{-1}B\vec d_i$, with shifts $s_i$, a block vector $B$ and tangential directions $\vec d_i$. In the following definition, the tangential directions play the same role, with $B=[\vec v_1,\vec v_2]$. In addition, the vectors $\vec v_1$ and $\vec v_2$ themselves belong to the subspace, which corresponds to the constant vectors $\vec e_1$ and $\vec e_2$ in $\mathcal{R}_m$. The resulting subspace is the one used by Druskin et al.\ for the solution of Lyapunov equations~\cite[Eq.~(6.3)]{DruskinSimonciniZaslavsky2014}, and the tangential rational Arnoldi algorithm of Section~\ref{subsec:arnoldi} is similar to their algorithm, which, however, chooses the poles and the tangential directions adaptively.}

\begin{definition}{(Tangential Rational Krylov Subspace)}
\rev{
    Consider a matrix $A\in \C^{n\times n}$, two vectors $\vec v_1, \vec v_2 \in \C^n$ and poles $\zeta_1,\ldots,\zeta_{m-2}$ with tangential directions $\vec d_1,\ldots,\vec d_{m-2}$ disjoint from the spectrum of $A$, as in Definition~\ref{def:rationalvectors}. We define $\mathcal{\vec Q}_1(A,[\vec v_1,\vec v_2]) = \operatorname{span}\{\vec v_1\}$ and the tangential rational Krylov subspace
    \begin{equation*}
        \mathcal{\vec Q}_m(A,[\vec v_1,\vec v_2]) = \operatorname{range}\left[\vec v_1,\; \vec v_2,\; (A-\zeta_1I)^{-1}[\vec v_1,\vec v_2]\vec d_1,\; \ldots,\; (A-\zeta_{m-2}I)^{-1}[\vec v_1,\vec v_2]\vec d_{m-2}\right], \quad m \geq 2.
    \end{equation*}}
\end{definition}
The connection with the vector-valued rational functions in Definition~\ref{def:rationalvectors} follows from the fact that \rev{$(A-\zeta_kI)^{-1}[\vec v_1,\vec v_2]\vec d_k = d_{k,1}(A-\zeta_kI)^{-1}\vec v_1 + d_{k,2}(A-\zeta_kI)^{-1}\vec v_2$, so that} every vector $\vec x \in \mathcal{\vec Q}_m(A,[\vec v_1,\vec v_2])$ can be written as
\begin{equation} \label{eq:isometry}
    \vec x = \phi_1(A)\vec v_1 + \phi_2(A)\vec v_2, \qquad \vec \phi(z) = \begin{bmatrix}
        \phi_1(z) \\
        \phi_2(z)
    \end{bmatrix} \in \mathcal{R}_{m},
\end{equation}
where $\mathcal{R}_m \subseteq \mathcal{R}_n$, $m \leq n$, is the space of vector-valued rational functions defined by the same poles \rev{and tangential directions}.
Now consider the discrete inner product $\langle \cdot, \cdot\rangle_{\mathcal{R}_n}$ from Definition~\ref{def:innerprod}, with nodes $z_1,\ldots,z_n$ and weight matrix $W\in \C^{n\times 2}$, whose rows are denoted by $\vec w^H_i$ (cf.\ Section~\ref{sec:problemformulation}). We collect the nodes in the diagonal matrix $Z=\operatorname{diag}(z_1,\ldots,z_n)$, as in Problem~\ref{problem:IEP}, and denote the columns of $W$ by $\vec \omega_1$ and $\vec \omega_2$, i.e., $W = [\vec \omega_1, \vec \omega_2]$. Taking $A = Z$ and $[\vec v_1,\vec v_2] = W$ in~\eqref{eq:isometry} gives an isometry between $\mathcal{R}_n$ and the tangential rational Krylov subspace, equipped with the standard Euclidean inner product $\langle \cdot, \cdot\rangle$.
\begin{proposition}\label{prop:isometry}
    With $Z$ and $W = [\vec \omega_1, \vec \omega_2]$ as above, the correspondence~\eqref{eq:isometry} is an isometry between $\mathcal{R}_{n}$, equipped with $\langle\cdot,\cdot\rangle_{\mathcal{R}_n}$, and $\mathcal{\vec Q}_n(Z,[\vec \omega_1,\vec \omega_2])$, equipped with $\langle\cdot,\cdot\rangle$.
\end{proposition}
\begin{proof}
    Let $\vec x,\vec y \in \mathcal{\vec Q}_n(Z,[\vec \omega_1,\vec \omega_2])$, then there exist $\vec \phi(z), \vec \psi(z) \in \mathcal{R}_{n}$ such that $\vec x = \phi_1(Z)\vec \omega_1 + \phi_2(Z)\vec \omega_2$ and $\vec y = \psi_1(Z)\vec \omega_1 + \psi_2(Z)\vec \omega_2$, cf.~\eqref{eq:isometry}. Since the $i$th row of $W$ is $\vec w_i^H$, the $i$th entries of $\vec x$ and $\vec y$ are $\vec w_i^H\vec \phi(z_i)$ and $\vec w_i^H\vec \psi(z_i)$, respectively. Taking the inner product gives
    \begin{equation*}
        \langle \vec x, \vec y\rangle = \vec y^H\vec x = \sum_{i=1}^{n}\overline{\vec w_i^H\vec \psi(z_i)}\,\vec w_i^H \vec \phi(z_i) = \sum_{i=1}^{n} \vec \psi(z_i)^H\vec w_i\vec w_i^H\vec \phi(z_i) = \langle \vec \phi(z), \vec \psi(z)\rangle_{\mathcal{R}_n}.
    \end{equation*}
\end{proof}
\subsection{Tangential rational Arnoldi algorithm}\label{subsec:arnoldi}
Proposition~\ref{prop:isometry} indicates that we can solve Problem~\ref{problem:general}, as defined for $\mathcal{R}_n$, by finding an orthonormal basis for the isometric space $\mathcal{\vec Q}_n(Z,[\vec \omega_1,\vec \omega_2])$. In this section, we explain how an Arnoldi-like algorithm constructs such an orthonormal basis for $\mathcal{\vec Q}_m(A,[
            \vec v_1,\vec v_2])$ with $A\in \C^{n\times n}$, $\vec v_1,\vec v_2\in\C^n$ and poles $\zeta_i = \nu_i/\mu_i$ \rev{with tangential directions $\vec d_i$}, $i=1,\ldots,m-2$.
An orthonormal basis for $\mathcal{\vec Q}_{k+1}(A,[
            \vec v_1,\vec v_2])$ can inductively be constructed given an orthonormal basis $\{\vec q_1, \ldots, \vec q_k\}$ of $\mathcal{\vec Q}_k(A,[
            \vec v_1,\vec v_2])$. Consider the vector
\begin{equation}\label{eq:orthvector}
    \widehat{\vec q}_{k+1} = \left(\mu_{k-1}A-\nu_{k-1}I\right)^{-1}\sum_{i=1}^k\left(\rho_{i,k}A - \eta_{i,k}I\right)\vec q_i \coloneqq \left(\mu_{k-1}A-\nu_{k-1}I\right)^{-1}\vec u_{k}.
\end{equation}
The \emph{continuation parameters} $\{\rho_{i,k},\eta_{i,k}\}$ should be carefully chosen such that the vector $\widehat{\vec q}_{k+1}$ expands the preceding Krylov space, i.e., $\widehat{\vec q}_{k+1} \in \mathcal{\vec Q}_{k+1}(A,[
            \vec v_1,\vec v_2])\setminus \mathcal{\vec Q}_k(A,[
            \vec v_1,\vec v_2])$.
The new vector $\vec q_{k+1}$ is then obtained by orthonormalizing
\begin{equation}\label{eq:orthonormalizing}
    \gamma_{k+1,k}\vec q_{k+1} = \widehat{\vec q}_{k+1} - \sum_{i=1}^k \gamma_{i,k} \vec q_i,
\end{equation}
against the already known orthonormal vectors $\vec q_1, \ldots, \vec q_k$.
The orthonormal vectors $\vec q_1, \ldots, \vec q_n$ form a basis of $\C^n$, so the Krylov space cannot be expanded further. Hence, for $k=n$ and $k=n+1$, the vector $\widehat{\vec q}_{k+1}$ already lies in $\operatorname{span}\{\vec q_1, \ldots, \vec q_n\}$, and equations~\eqref{eq:orthvector} and \eqref{eq:orthonormalizing} only involve $\vec q_1, \ldots, \vec q_n$, i.e., $\gamma_{n+1,n} = \gamma_{n+2,n+1} = 0$. Combining equations~\eqref{eq:orthvector} and \eqref{eq:orthonormalizing} for $k=2,\ldots,n+1$ yields
\begin{align*}
\begin{split}
    A\begin{bmatrix}
        \vec q_1, \ldots, \vec q_n
    \end{bmatrix} K_n &= \begin{bmatrix}
        \vec q_1, \ldots, \vec q_n
    \end{bmatrix} H_n, \\
    K_n &= \begin{bmatrix}
        \mu_1\gamma_{12} - \rho_{12} &\cdots & \mu_{n}\gamma_{1,n+1} - \rho_{1,n+1}\\
        \mu_{1}\gamma_{22} - \rho_{22} & \cdots & \mu_{n}\gamma_{2,n+1} - \rho_{2,n+1}\\
        \mu_{1}\gamma_{32} & \ddots & \vdots \\
                      & \ddots  & \mu_{n}\gamma_{n-1,n+1} - \rho_{n-1,n+1}\\
                      & & \mu_{n}\gamma_{n,n+1} - \rho_{n,n+1}
    \end{bmatrix} \in \C^{n\times n},\\
    H_n &= \begin{bmatrix}
        \nu_1\gamma_{12} - \eta_{12} &\cdots & \nu_{n}\gamma_{1,n+1} - \eta_{1,n+1}\\
        \nu_{1}\gamma_{22} - \eta_{22} & \cdots & \nu_{n}\gamma_{2,n+1} - \eta_{2,n+1}\\
        \nu_{1}\gamma_{32} & \ddots & \vdots \\
                      & \ddots  & \nu_{n}\gamma_{n-1,n+1} - \eta_{n-1,n+1}\\
                      & & \nu_{n}\gamma_{n,n+1} - \eta_{n,n+1}
    \end{bmatrix} \in \C^{n\times n},
\end{split}
\end{align*}
where $(H_n,K_n)$ has the same 2-Hessenberg pencil structure as in Problem~\ref{problem:IEP}.

We have not yet specified how to choose the continuation parameters $\rho_{i,k}$ and $\eta_{i,k}$. \rev{Let $\zeta_{k-1} \in \mathbb{C}$ be the pole to be added, with tangential direction $\vec d_{k-1} = [a_{k-1}, b_{k-1}]^T$. We choose $\mu_{k-1} = 1$ and $\nu_{k-1} = \zeta_{k-1}$, and set $\rho_{i,k} = \eta_{i,k} = 0$ for all $i$, except for $\eta_{1,k}$ and $\eta_{2,k}$, which are determined by $\vec u_k = -\eta_{1,k}\vec q_1 - \eta_{2,k}\vec q_2 = c(a_{k-1}\vec v_1 + b_{k-1}\vec v_2)$, with $c\in\C$ such that $\|\vec u_k\|_2 = 1$. The isometry in equation~\eqref{eq:isometry} gives the correspondence}
\begin{equation*}
    \rev{\widehat{\vec q}_{k+1} = c(A-\zeta_{k-1}I)^{-1}(a_{k-1}\vec v_1 + b_{k-1}\vec v_2) \longleftrightarrow \vec \phi(z) = c\,\frac{\vec d_{k-1}}{z-\zeta_{k-1}}.}
\end{equation*}

Theoretically, this procedure should give an orthonormal set $\{\vec q_1, \dots, \vec q_{k+1}\}$. 
In practice, finite precision arithmetic causes a loss of orthogonality due to accumulated round-off errors. To correct this, we employ \emph{full} re-orthogonalization~\cite{Giraud2005}, in which the current vector is orthogonalized against \emph{all} the previous vectors a second time, cf.~\eqref{eq:orthonormalizing}.
To reduce computational cost, one could use \emph{partial} re-orthogonalization~\cite{Simon1984}, which only re-orthogonalizes when an estimate of the loss of orthogonality exceeds a threshold, or \emph{selective} re-orthogonalization~\cite{ParlettScott1979}, which only orthogonalizes against the converged Ritz vectors.

\begin{remark}[Poles at infinity]\label{rem:infpoles}
\rev{Instead of adding only finite poles, one can also add poles at infinity, which corresponds to adding polynomial terms. Adding $p$ poles at infinity in the first component extends $\mathcal{R}_m$ with $z\vec e_1, z^2\vec e_1, \ldots, z^p\vec e_1$. By~\eqref{eq:isometry}, $z^\ell\vec e_1$ corresponds to $A^\ell\vec v_1$, so the tangential rational Krylov subspace is extended to
\begin{equation*}
    \mathcal{\vec Q}_m(A,[\vec v_1,\vec v_2]) + \operatorname{span}\{A\vec v_1, A^2\vec v_1, \ldots, A^p\vec v_1\}.
\end{equation*}
In the tangential rational Arnoldi algorithm, a pole at infinity corresponds to $\mu_{k-1}=0$ and $\nu_{k-1}=-1$ in~\eqref{eq:orthvector}. For the first pole at infinity, we set all continuation parameters to zero, except $\rho_{1,k}=1$. The isometry in equation~\eqref{eq:isometry} gives the correspondence
\begin{equation*}
    \widehat{\vec q}_{k+1} = A\vec q_1 \longleftrightarrow \vec\phi(z) = z\,\vec\phi_1(z),
\end{equation*}
which adds the term $z\vec e_1$, since $\vec\phi_1 \in \operatorname{span}\{\vec e_1\}$. When adding the $\ell$th pole at infinity, $\ell \geq 2$, let $\vec\phi_j$ be the basis function generated by the previous one, whose first component contains $z^{\ell-1}$. We set all continuation parameters to zero, except $\rho_{j,k}=1$ and $\rho_{2,k} = -\phi_{j,2}(\infty)/\phi_{2,2}(\infty)$, where $\phi_{i,2}(\infty) = \lim_{z\to\infty}\phi_{i,2}(z)$. This gives the correspondence
\begin{equation*}
    \widehat{\vec q}_{k+1} = A\left(\vec q_j + \rho_{2,k}\vec q_2\right) \longleftrightarrow \vec\phi(z) = z\left(\vec\phi_j(z) + \rho_{2,k}\vec\phi_2(z)\right).
\end{equation*}
The choice of $\rho_{2,k}$ ensures that the second component of $\vec\phi_j + \rho_{2,k}\vec\phi_2$ vanishes at infinity, so that $\vec\phi$ does not contain the term $z\vec e_2$, while its first component contains $z^{\ell-1}$, so that $\vec\phi$ contains the new term $z^\ell\vec e_1$. The resulting pencil has $k_{k+1,k-1}=0$ and $h_{k+1,k-1}=-\gamma_{k+1,k}\neq 0$, i.e., the pole at infinity appears on the second subdiagonal.}
\end{remark}

\section{Numerical Experiments}\label{sec:numexp}
To compare our algorithms, we adapt the two numerical experiments of Van Buggenhout et al.~\cite{VanBuVanBaVand2022} to the vector case. The first experiment uses equidistant nodes on the complex unit circle and equidistant poles on a circle of radius $3/2$. We assign weights $\vec w_{i}$ uniformly distributed within the complex square $\left(\frac{1}{2}, \frac{3}{2}\right) \times\left(\frac{1}{2}i,\frac{3}{2}i\right)$ and \rev{assign to each pole a random tangential direction $\vec d_k = \vec g_k/\|\vec g_k\|_2$, with $\vec g_k = \vec x_k + i\vec y_k$, where $\vec x_k, \vec y_k \sim \mathcal{N}(\vec 0, I_2)$ are independent}.

The following four error metrics are used to compare our algorithms:
\begin{itemize}
    \item orthonormality of $Q$,
    \begin{equation*}
        \operatorname{err}_Q = \|Q^HQ-I_n\|_2,
    \end{equation*}
    \item orthonormality of $\vec \phi_j$'s,
    \begin{equation*}
        \operatorname{err}_{\vec \phi}=\|M- I_n\|_2,\; \text{with}\; M_{i,j} = \langle \vec \phi_i, \vec \phi_j\rangle_{\mathcal{R}_n},
    \end{equation*}
    \item relative accuracy of the poles,
    \begin{equation*}
        \operatorname{err}_\zeta = \max_{1 \leq i \leq n-2} \Bigg| \frac{h_{i+2,i}}{k_{i+2,i}} -\zeta_i\Bigg|/ |\zeta_i|
    \end{equation*}
    \item and the relative accuracy of the recurrence relation,
    \begin{equation*}
        \operatorname{err}_{r} = \frac{\|ZQK - QH\|_2}{\max\{\|ZQK\|_2,\|QH\|_2\}}.
    \end{equation*}
\end{itemize}

\begin{figure}
  \centering
\definecolor{cup}{rgb}{0.00,0.20,0.60}
\definecolor{ckryl}{rgb}{0.95,0.60,0.00}
\begin{tikzpicture}
\begin{groupplot}[
    group style={group size=2 by 1, horizontal sep=1.6cm},
    ymode=log,
    scale only axis,
    width=0.5\linewidth-1.7cm, height=0.35\linewidth,
    xlabel={$n$},
    xmin=0, xmax=150,
    grid=major,
    legend cell align=left,
    legend style={font=\scriptsize, inner sep=1pt, row sep=-2pt,
        fill=white, fill opacity=0.9, text opacity=1},
    legend image code/.code={
        \draw[mark repeat=2, mark phase=2, mark size=1.5pt]
            plot coordinates {(0cm,0cm) (0.25cm,0cm) (0.5cm,0cm)};},
    table/x=n,
    error bars/y dir=both, error bars/y explicit,
    error bars/error bar style={solid},
]
\nextgroupplot[ylabel={error},
    legend pos=north west,
    ytick={1e-16,1e-12,1e-8,1e-4,1e0},
    minor ytick={1e-17,1e-15,1e-14,1e-13,1e-11,1e-10,1e-9,1e-7,1e-6,1e-5,1e-3,1e-2,1e-1,1e1,1e2}]
\addplot[cup, solid, mark=*]
    table[y=Q_up, y error minus=Q_up_minus, y error plus=Q_up_plus] {experiment1.dat};
\addlegendentry{$\mathrm{err}_Q$ updating}
\addplot[ckryl, dashed, mark=*, mark options={solid}]
    table[y=Q_kryl, y error minus=Q_kryl_minus, y error plus=Q_kryl_plus] {experiment1.dat};
\addlegendentry{$\mathrm{err}_Q$ Krylov}
\addplot[cup, solid, mark=square]
    table[y=phi_up, y error minus=phi_up_minus, y error plus=phi_up_plus] {experiment1.dat};
\addlegendentry{$\mathrm{err}_{\phi}$ updating}
\addplot[ckryl, dashed, mark=square, mark options={solid}]
    table[y=phi_kryl, y error minus=phi_kryl_minus, y error plus=phi_kryl_plus] {experiment1.dat};
\addlegendentry{$\mathrm{err}_{\phi}$ Krylov}

\nextgroupplot[ylabel={relative error},
    legend pos=north west,
    ymax=1e-11]
\addplot[cup, solid, mark=square]
    table[y=poles_up, y error minus=poles_up_minus, y error plus=poles_up_plus] {experiment1.dat};
\addlegendentry{$\mathrm{err}_{\zeta}$ updating}
\addplot[ckryl, dashed, mark=square, mark options={solid}]
    table[y=poles_kryl, y error minus=poles_kryl_minus, y error plus=poles_kryl_plus] {experiment1.dat};
\addlegendentry{$\mathrm{err}_{\zeta}$ Krylov}
\addplot[cup, solid, mark=*]
    table[y=rec_up, y error minus=rec_up_minus, y error plus=rec_up_plus] {experiment1.dat};
\addlegendentry{$\mathrm{err}_r$ updating}
\addplot[ckryl, dashed, mark=*, mark options={solid}]
    table[y=rec_kryl, y error minus=rec_kryl_minus, y error plus=rec_kryl_plus] {experiment1.dat};
\addlegendentry{$\mathrm{err}_r$ Krylov}
\end{groupplot}
\end{tikzpicture}
  \caption{Left: orthonormality errors of $Q$ and $\phi$. Right: relative errors of the
    recurrence relation and the poles. Equidistant nodes on the unit circle. Marks: mean over 30 runs, error bars: min to max.}
  \label{fig:experiment1}
\end{figure}

Figure \ref{fig:experiment1} shows the different error measures for both algorithms for increasing problem sizes. The errors are averaged over \rev{30} runs because we use random weights and \rev{random tangential directions $\vec d_k$}.
It follows that both algorithms compute a matrix $Q$ which is numerically close to being unitary. However, the orthonormality for the $\phi_j$'s is much worse for the Krylov algorithm; an exponential loss of orthogonality is observed as $n$ increases. A similar error behavior for $\operatorname{err}_{\vec \phi}$ is observed in the experiments for the non-vector case~\cite{VanBuVanBaVand2022}. They explained this by looking at the condition number of the underlying linear system for evaluating the rational functions, necessary to compute the matrix $M$. The updating algorithm uses unitary similarity transformations; consequently, when nodes are chosen on the unit circle, the associated pencil consists of unitary matrices. This results in a well-conditioned system for evaluating $\vec \phi_j(z)$~\eqref{eq:linearsystem}, in contrast to the system obtained by the Krylov method. The errors $\operatorname{err}_{\zeta}$ and $\operatorname{err}_{r}$ for the Krylov algorithm are close to machine precision. In contrast, we observe larger errors for the updating algorithm. 

In the second experiment, we show the effect of having two nodes really close to each other, and the corresponding weights being linearly dependent. More concretely, we consider $n$ equidistant nodes on the unit circle except for the $\kappa$th node and $(\kappa+1)$th node, which are located at a small angle $\theta=10^{-6}$ from each other. We again take the same \rev{poles on the circle of radius $3/2$, random tangential directions $\vec d_k$, and} uniformly random weights as in the first experiment, except with the extra condition that $\vec w_{\kappa}$ and $\vec w_{\kappa+1}$ are linearly dependent. Intuitively, this means that our inner product with $n$ nodes is close to an inner product with $n-1$ nodes, implying that the $n$th orthogonal vector-valued rational function $\vec \phi_n$ does not extend our orthogonal basis. This explains the jump in the orthogonality error $\operatorname{err}_{\vec \phi}$ for the updating algorithm
at $n=\kappa = 40$, as can be seen in Figure \ref{fig:experiment2}. \rev{Additional experiments with $\theta$ ranging from $10^{-1}$ to $10^{-12}$ show that the size of the jump is inversely proportional to $\theta$, with $\operatorname{err}_{\vec \phi}$ behaving like $10^{-13}/\theta$ after the jump.}

\begin{figure}
    \centering
\definecolor{cup}{rgb}{0.00,0.20,0.60}
\definecolor{ckryl}{rgb}{0.95,0.60,0.00}
\begin{tikzpicture}
\begin{groupplot}[
    group style={group size=2 by 1, horizontal sep=1.6cm},
    ymode=log,
    scale only axis,
    width=0.5\linewidth-1.7cm, height=0.35\linewidth,
    xlabel={$n$},
    xmin=0, xmax=150,
    grid=major,
    legend cell align=left,
    legend style={font=\scriptsize, inner sep=1pt, row sep=-2pt,
        fill=white, fill opacity=0.9, text opacity=1},
    legend image code/.code={
        \draw[mark repeat=2, mark phase=2, mark size=1.5pt]
            plot coordinates {(0cm,0cm) (0.25cm,0cm) (0.5cm,0cm)};},
    table/x=n,
    error bars/y dir=both, error bars/y explicit,
    error bars/error bar style={solid},
]
\nextgroupplot[ylabel={error},
    legend pos=north west,
    ytick={1e-16,1e-12,1e-8,1e-4,1e0,1e4,1e8},
    minor ytick={1e-17,1e-15,1e-14,1e-13,1e-11,1e-10,1e-9,1e-7,1e-6,1e-5,1e-3,1e-2,1e-1,1e1,1e2,1e3,1e5,1e6,1e7,1e9}]
\addplot[cup, solid, mark=*]
    table[y=Q_up, y error minus=Q_up_minus, y error plus=Q_up_plus] {experiment2.dat};
\addlegendentry{$\mathrm{err}_Q$ updating}
\addplot[ckryl, dashed, mark=*, mark options={solid}]
    table[y=Q_kryl, y error minus=Q_kryl_minus, y error plus=Q_kryl_plus] {experiment2.dat};
\addlegendentry{$\mathrm{err}_Q$ Krylov}
\addplot[cup, solid, mark=square]
    table[y=phi_up, y error minus=phi_up_minus, y error plus=phi_up_plus] {experiment2.dat};
\addlegendentry{$\mathrm{err}_{\phi}$ updating}
\addplot[ckryl, dashed, mark=square, mark options={solid}]
    table[y=phi_kryl, y error minus=phi_kryl_minus, y error plus=phi_kryl_plus] {experiment2.dat};
\addlegendentry{$\mathrm{err}_{\phi}$ Krylov}

\nextgroupplot[ylabel={relative error},
    legend pos=north west,
    ymax=1e-6,
    ytick={1e-16,1e-14,1e-12,1e-10,1e-8,1e-6},
    minor ytick={1e-17,1e-15,1e-13,1e-11,1e-9,1e-7}]
\addplot[cup, solid, mark=square]
    table[y=poles_up, y error minus=poles_up_minus, y error plus=poles_up_plus] {experiment2.dat};
\addlegendentry{$\mathrm{err}_{\zeta}$ updating}
\addplot[ckryl, dashed, mark=square, mark options={solid}]
    table[y=poles_kryl, y error minus=poles_kryl_minus, y error plus=poles_kryl_plus] {experiment2.dat};
\addlegendentry{$\mathrm{err}_{\zeta}$ Krylov}
\addplot[cup, solid, mark=*]
    table[y=rec_up, y error minus=rec_up_minus, y error plus=rec_up_plus] {experiment2.dat};
\addlegendentry{$\mathrm{err}_r$ updating}
\addplot[ckryl, dashed, mark=*, mark options={solid}]
    table[y=rec_kryl, y error minus=rec_kryl_minus, y error plus=rec_kryl_plus] {experiment2.dat};
\addlegendentry{$\mathrm{err}_r$ Krylov}
\end{groupplot}
\end{tikzpicture}
    \caption{Left: orthonormality errors of $Q$ and $\phi$. Right: relative errors of the recurrence relation and the poles. Equidistant nodes on the unit circle, except for two of them close to each other. Marks: mean over 30 runs, error bars: min to max.}
    \label{fig:experiment2}
\end{figure}

\section{Connection to rational approximation}\label{sec:application}
\rev{In this section, we show that two well-known rational approximation methods fit in the framework of vector-valued rational functions: the lightning method~\cite{GopalTrefethen2019,HerremansHuybrechsTrefethen2023} and the AAA algorithm~\cite{NakatsukasaSeteTrefethen2018}. Both approximate a function $f$ from samples $f_i = f(z_i)$ in $M$ pairwise distinct sample points $z_1,\ldots,z_M$. They differ only in the choice of the poles, the tangential directions and the nodes.}

\subsection{Rational approximants as vector-valued rational functions}\label{subsec:framework}
\rev{We collect the numerator and denominator of an approximant $r = n/d$ in the vector $[n, d]^T \in \mathcal{R}_{m+2}$, whose $m$ poles $\zeta_1,\ldots,\zeta_m$ have tangential directions $\vec d_k = [a_k, b_k]^T$. By Definition~\ref{def:rationalvectors}, this means that
\begin{equation}\label{eq:genbary}
    r(z) = \frac{n(z)}{d(z)} = \frac{\beta_1 + \sum_{k=1}^{m} \dfrac{\alpha_k a_k}{z-\zeta_k}}{\beta_2 + \sum_{k=1}^{m} \dfrac{\alpha_k b_k}{z-\zeta_k}},
\end{equation}
with free parameters $\beta_1$, $\beta_2$ and $\alpha_1, \ldots, \alpha_m$. If $\alpha_k \neq 0$, then $r$ takes the value $a_k/b_k$ at $\zeta_k$.

Let $I \subseteq \{1,\ldots,M\}$ index the sample points that are used as nodes, and take the weights $\vec w_i = [1, -\overline{f_i}]^T$ for $i \in I$. Then
\begin{equation}\label{eq:linres}
    \vec w_i^H \begin{bmatrix} n(z_i) \\ d(z_i) \end{bmatrix} = n(z_i) - f_i\, d(z_i) = d(z_i)\left(r(z_i) - f_i\right),
\end{equation}
so that the norm $\|\cdot\|_I$ induced by the inner product of Definition~\ref{def:innerprod} with these nodes and weights satisfies
\begin{equation}\label{eq:linearizedleastsquares}
    \left\| \begin{bmatrix} n(z) \\ d(z) \end{bmatrix} \right\|_I^2 = \sum_{i\in I} \left| n(z_i) - f_i\, d(z_i) \right|^2 = \sum_{i\in I} |d(z_i)|^2 \left| r(z_i) - f_i \right|^2.
\end{equation}
Minimizing this norm is the linearization of the least-squares problem for $r$. Assume that the inner product is positive definite on $\mathcal{R}_{m+2}$, which requires $|I| \geq m+2$, let $\vec \phi_1, \ldots, \vec \phi_{m+2}$ be the orthonormal basis of $\mathcal{R}_{m+2}$ computed by our algorithms, so that $\vec\phi_j \in \mathcal{R}_j$ for all $j$, and write $[n, d]^T = \sum_{j=1}^{m+2} c_j \vec \phi_j$, with $\vec c = [c_1, \ldots, c_{m+2}]^T$. By orthonormality, the linearized least-squares problem becomes
\begin{equation}\label{eq:coefproblem}
    \min_{\vec c \in \C^{m+2}} \left\| \begin{bmatrix} n(z) \\ d(z) \end{bmatrix} \right\|_I^2 = \min_{\vec c \in \C^{m+2}} \left\| \sum_{j=1}^{m+2} c_j \vec \phi_j(z) \right\|_I^2 = \min_{\vec c \in \C^{m+2}} \sum_{j=1}^{m+2} |c_j|^2.
\end{equation}
This problem is homogeneous in $\vec c$, so a normalization is needed to exclude the trivial solution $\vec c = \vec 0$. The methods below differ in this choice, and it is natural to state it in terms of the parameters of~\eqref{eq:genbary}, i.e., the coefficients $\vec x = [\beta_1, \beta_2, \alpha_1, \ldots, \alpha_m]^T$ of $[n, d]^T = \sum_{j=1}^{m+2} x_j \vec r_j$ in the basis of Definition~\ref{def:rationalvectors}. Since $\vec \phi_j \in \mathcal{R}_j \setminus \mathcal{R}_{j-1}$ (see Section~\ref{sec:problemformulation}), the two bases are related by
\begin{equation}\label{eq:changeofbasis}
    \begin{bmatrix} \vec r_1, & \ldots, & \vec r_{m+2} \end{bmatrix} = \begin{bmatrix} \vec \phi_1, & \ldots, & \vec \phi_{m+2} \end{bmatrix} R,
\end{equation}
with $R \in \C^{(m+2)\times(m+2)}$ upper triangular and nonsingular. Hence $\vec c = R\vec x$, and the objective of~\eqref{eq:coefproblem} equals $\|R\vec x\|_2^2$. Each method minimizes this objective subject to its own normalization, which is a constraint on $\vec c$ or on $\vec x$.

The IEP of Problem~\ref{problem:IEP} is square, so $|I|$ nodes correspond to $|I|-2$ poles. When $|I| > m+2$, the updating algorithm therefore needs additional poles, distinct from the nodes. Since the spaces are nested, these do not influence $\vec \phi_1, \ldots, \vec \phi_{m+2}$, and they do not enter the approximant.}

\subsection{The lightning method}\label{subsec:lightning}
The lightning method of Gopal and Trefethen~\cite{GopalTrefethen2019} approximates functions with branch point singularities by rational functions with preassigned poles that are exponentially clustered near the singularities, so that the coefficients follow from a linear least-squares problem. Herremans, Huybrechs and Trefethen~\cite{HerremansHuybrechsTrefethen2023} add a polynomial term, which gives the rational + polynomial approximant
\begin{equation}\label{eq:ratapprox1}
    r(z) = \sum_{j=1}^{N_1}\frac{a_j}{z-\tau_j} + \sum_{j=0}^{N_2}b_jz^j,
\end{equation}
with tapered exponentially clustered poles towards zero
\begin{equation*}
    \tau_j = -C\exp\left(-\sigma\left(\sqrt{N_1} - \sqrt{j}\right)\right), \quad j = 1,2,\ldots,N_1,
\end{equation*}
and determine the coefficients by minimizing $\sum_{i=1}^M |r(z_i) - f_i|^2$. For $f(z) = \sqrt{z}$, they show that this approximant attains the optimal minimax convergence rate $\mathcal{O}(\exp(-\pi\sqrt{2(N_1+N_2)}))$~\cite{Stahl2003}, with $C=2$ and $\sigma = 2\sqrt 2 \pi$. \rev{Without the tapering and the polynomial part, the convergence is slower: with uniformly clustered poles
\begin{equation}\label{eq:uniformpoles}
    \tau_j = -C\exp\left(-\sigma\,\frac{j-1}{\sqrt{N_1}}\right), \quad j = 1,2,\ldots,N_1,
\end{equation}
with $C = 2$ and $\sigma = \pi$, and only the constant term of the polynomial part, i.e., $N_2 = 0$, the error decreases at the slower empirical rate $\mathcal{O}(\exp(-\pi\sqrt{0.3N_1}))$~\cite{HerremansHuybrechsTrefethen2023}.} Although the resulting least-squares system of the lightning method is severely ill-conditioned, one still finds a good approximation using oversampling and regularization~\cite{HerremansHuybrechsTrefethen2023,AdcockHuybrechs2019}.

\rev{In our framework, all sample points are nodes and every pole has tangential direction $\vec e_1$. The polynomial part of~\eqref{eq:ratapprox1} corresponds to $N_2$ poles at infinity in the first component, i.e., the terms $z^\ell \vec e_1$, $\ell = 1,\ldots,N_2$, which are added as in Remark~\ref{rem:infpoles}. Hence $m = N_1 + N_2$. No pole enters the second component, so $d(z) = \beta_2$ is constant, and we normalize $\beta_2 = 1$, i.e., $x_2 = 1$, where $\vec x$ now also contains the coefficients of the terms $z^\ell \vec e_1$. Then $[n, d]^T = [r, 1]^T$ and, by~\eqref{eq:linearizedleastsquares}, the linearized objective is exactly the least-squares objective of the lightning method.

For the same reason, the second component $\phi_{j,2}$ of every basis function $\vec \phi_j$, $j \leq m+2$, is constant. In particular, $\phi_{j,2}(\infty) = \phi_{j,2}$ in the continuation parameters of Remark~\ref{rem:infpoles}, and the normalization is the linear constraint $\sum_{j=1}^{m+2} c_j \phi_{j,2} = 1$. Solving~\eqref{eq:coefproblem} subject to this constraint gives
\begin{equation*}
    c_j = \frac{\overline{\phi_{j,2}}}{\sum_{k=1}^{m+2} |\phi_{k,2}|^2}, \quad j = 1,\ldots,m+2.
\end{equation*}}

\subsection{The AAA algorithm}\label{subsec:aaa}
\rev{The AAA algorithm~\cite{NakatsukasaSeteTrefethen2018} uses the barycentric form
\begin{equation}\label{eq:barycentric}
    r(z) = \frac{n(z)}{d(z)}, \qquad n(z) = \sum_{k=1}^m \frac{\alpha_k f(\zeta_k)}{z - \zeta_k}, \qquad d(z) = \sum_{k=1}^m \frac{\alpha_k}{z-\zeta_k},
\end{equation}
where the support points $\zeta_1, \ldots, \zeta_m$ are selected from the sample points. This is~\eqref{eq:genbary} with $\beta_1 = \beta_2 = 0$ and $\vec d_k = [f(\zeta_k), 1]^T$, and the value $a_k/b_k$ at $\zeta_k$ is the interpolation condition $r(\zeta_k) = f(\zeta_k)$. The support points are poles of the vector $[n, d]^T$, but not of $r$. The poles of $r$ are the zeros of $d$, which depend on the weights $\alpha_k$. In contrast to the lightning method, they are therefore not prescribed.

Since $r$ already interpolates $f$ in the support points, only the remaining sample points are used as nodes, $I_m = \{ i : z_i \notin \{\zeta_1,\ldots,\zeta_m\}\}$, in agreement with the requirement in Definition~\ref{def:innerprod} that nodes and poles are distinct.

The normalization used in AAA is $\|\vec\alpha\|_2 = 1$. Denote the last $m$ columns of $R$ by $\widehat R \in \C^{(m+2)\times m}$. Since $\beta_1 = \beta_2 = 0$, we have $\vec c = \widehat R\vec\alpha$, and~\eqref{eq:coefproblem} becomes
\begin{equation}\label{eq:aaaproblem}
    \min_{\|\vec\alpha\|_2 = 1} \left\| \begin{bmatrix} n(z) \\ d(z) \end{bmatrix} \right\|_{I_m} = \min_{\|\vec\alpha\|_2 = 1} \|\widehat R\vec\alpha\|_2.
\end{equation}
Its solution $\vec \alpha$ is the right singular vector of the small matrix $\widehat R$ belonging to its smallest singular value.

Problem~\eqref{eq:aaaproblem} is exactly the least-squares problem of AAA~\cite{NakatsukasaSeteTrefethen2018}. To see this, apply the map $\vec \psi \mapsto [\vec w_i^H\vec\psi(z_i)]_{i\in I_m}$ to both sides of~\eqref{eq:changeofbasis}. It maps the orthonormal functions $\vec\phi_1, \ldots, \vec\phi_{m+2}$ to orthonormal vectors, namely the first $m+2$ columns $Q_{m+2}$ of the matrix $Q$ of the IEP. Moreover, $\vec w_i^H \vec r_1(z_i) = 1$, $\vec w_i^H \vec r_2(z_i) = -f_i$, and $\vec w_i^H \vec r_{k+2}(z_i) = (f(\zeta_k) - f_i)/(z_i - \zeta_k)$. Hence
\begin{equation}\label{eq:loewnerqr}
    \begin{bmatrix} \vec 1, & -\vec f, & -L \end{bmatrix} = Q_{m+2} R, \qquad L_{ik} = \frac{f_i - f(\zeta_k)}{z_i - \zeta_k},
\end{equation}
where $\vec 1$ is the vector of all ones, $\vec f = [f_i]_{i\in I_m}$, and $L$ is the Loewner matrix of AAA. The change of basis $R$ is thus the triangular factor of a QR factorization of the Loewner matrix, augmented with the columns $\vec 1$ and $-\vec f$. In particular, $L = -Q_{m+2}\widehat R$, so $\|L\vec\alpha\|_2 = \|\widehat R\vec\alpha\|_2$, and~\eqref{eq:aaaproblem} is the AAA problem of minimizing $\|L\vec\alpha\|_2$ subject to $\|\vec\alpha\|_2 = 1$.

The matrix $\widehat R$ does not have to be computed from $L$; it is a by-product of the tangential rational Arnoldi algorithm of Section~\ref{subsec:arnoldi}. The input of this algorithm is the matrix $A = Z = \operatorname{diag}(z_i)_{i\in I_m}$, the starting vectors $\vec v_1 = \vec 1$ and $\vec v_2 = -\vec f$, and the poles $\zeta_j$ with tangential directions $\vec d_j = [f(\zeta_j), 1]^T$, $j = 1,\ldots,m$. The resulting orthonormal vectors $\vec q_1,\ldots,\vec q_{m+2}$ are the columns of $Q_{m+2}$. By~\eqref{eq:orthvector}, the step that adds the pole $\zeta_j$ forms
\begin{equation*}
    \widehat{\vec q}_{j+2} = s_j\,(Z-\zeta_j I)^{-1}\left(f(\zeta_j)\vec 1 - \vec f\right) = -s_j\, L\vec e_j,
\end{equation*}
where $s_j$ denotes the scaling $c$ of that step. Since $\mu_j = 1$ and $\rho_{i,j+1} = 0$, the $j$th column of $K$ contains the coefficients $\gamma_{i,j+1}$, so~\eqref{eq:orthonormalizing} reads $\widehat{\vec q}_{j+2} = Q_{m+2}K_{1:m+2,\,j}$. On the other hand, \eqref{eq:loewnerqr} gives $-L\vec e_j = Q_{m+2}\widehat R\vec e_j$. Comparing both expressions for $\widehat{\vec q}_{j+2}$ and using the orthonormality of $Q_{m+2}$, we obtain $K_{1:m+2,\,j} = s_j\widehat R\vec e_j$, i.e.,
\begin{equation*}
    \widehat R = K_{1:m+2,\,1:m}\operatorname{diag}(s_1,\ldots,s_m)^{-1}.
\end{equation*}
Hence, the AAA weights $\vec\alpha$ follow from an SVD of a column-scaled block of $K$, without forming $L$. As soon as $\widehat R$ is numerically rank deficient, i.e., the ratio of its largest to its smallest singular value exceeds $1/(3\varepsilon)$ with $\varepsilon$ the machine precision, we follow the implementation of AAA in Chebfun~\cite{Trefethen2014}. Let $D$ be the diagonal matrix containing the column norms of $\widehat R$, which by~\eqref{eq:loewnerqr} are also the column norms of $L$. We take $\vec\alpha = D^{-1}\vec v$, where $\vec v$ is the right singular vector of $\widehat R D^{-1}$ belonging to its smallest singular value. This $\vec\alpha$ minimizes the same objective $\|\widehat R\vec\alpha\|_2$ subject to the normalization $\|D\vec\alpha\|_2 = 1$. Since a scaling of $\vec\alpha$ does not change $r$, we afterwards normalize it such that $\|\vec\alpha\|_2 = 1$.

In step $m$, AAA selects as new support point $\zeta_m$ the sample point $z_i$, $i\in I_{m-1}$, for which $|r(z_i) - f_i|$ is largest, with $r$ the approximant of the previous step. The new approximant interpolates $f$ in $\zeta_m$, so $z_i$ is removed from the nodes, $I_m = I_{m-1}\setminus\{i\}$, and $\zeta_m$ is added as a pole with direction $\vec d_m = [f(\zeta_m), 1]^T$. Because the inner product changes, the orthonormal vectors of the previous step, restricted to the nodes in $I_m$, are no longer orthonormal, and the decomposition cannot simply be extended by one step of the tangential rational Arnoldi algorithm. In our implementation, we therefore rerun the algorithm in every step. This costs $\mathcal{O}(|I_m|\,m^2)$ operations, which is of the same order as the SVD of the Loewner matrix in the standard implementation of AAA. It could be reduced by updating the QR factorization~\eqref{eq:loewnerqr} instead: removing the node $z_i$ deletes a row of $[\vec 1, -\vec f, -L]$ and adding the pole $\zeta_m$ appends a column, and both modifications can be handled by standard updating techniques~\cite[Section~6.5]{GolubVanLoan2013}. We do not pursue this here. Table~\ref{tab:parameters} summarizes the parameters of both methods.}

\begin{table}[ht]
\centering
\caption{Parameters of the framework for the lightning method and the AAA algorithm.}
\label{tab:parameters}
\begin{tabular}{lll}
    \hline
     & Lightning & AAA \\
    \hline
    Nodes $z_i$ & all sample points & sample points except support points \\
    Weights $\vec w_i$ & $[1, -\overline{f_i}]^T$ & $[1, -\overline{f_i}]^T$ \\
    Poles $\zeta_k$ & $\tau_j$ and $N_2$ poles at infinity & support points \\
    Directions $\vec d_k$ & $\vec e_1$ & $[f(\zeta_k), 1]^T$ \\
    Normalization & $\beta_2 = 1$, i.e., $d \equiv 1$ & $\beta_1 = \beta_2 = 0$, $\|\vec \alpha\|_2 = 1$ \\
    Constraint on & $\vec c$ (linear) & $\vec x$, via $\vec c = \widehat R\vec\alpha$ \\
    Poles of $r$ & prescribed & zeros of $d$, determined by $\vec \alpha$ \\
    \hline
\end{tabular}
\end{table}

\subsection{Numerical experiment}\label{subsec:ratapproxexp}
\rev{We approximate $f(z) = \sqrt z$ on $[0,1]$ with the lightning and AAA approximants, both computed in our framework. The sample points are $0$ and the $M{-}1$ points of an exponentially graded grid between $10^{-24}$ and $1$, i.e., in \texttt{MATLAB} notation: {\tt logspace(-24,0,M-1)}. For the lightning approximant, we take $N_1 = \text{round}(s^2)$ for $s = 1, 1.25, \ldots, 8$, $N_2 = \text{ceil}(2\sqrt{N_1})$, and $M = 21(N_1+N_2+2)$. For AAA, we take $M = 1722$, the value for the largest lightning approximant.

Both approximants are computed with the tangential rational Arnoldi algorithm with full re-orthogonalization, stopped after $m$ poles. For AAA, it is rerun in every step $m$ on the remaining sample points, and no stopping criterion is used. The lightning approximant is evaluated with~\eqref{eq:linearsystem} and the AAA approximant with the barycentric form~\eqref{eq:barycentric}. We do not use the updating algorithm here: on this grid the nodes are close together and their weights nearly linearly dependent, which is the situation of Figure~\ref{fig:experiment2}, and the computed poles lose their accuracy. This experiment serves as a stress test: the poles span many orders of magnitude and the sample points are exponentially graded. We do not aim to compete in computational cost with the dedicated implementations of these methods. The maximal error $\|f-r\|_\infty$ is approximated by sampling $|f-r|$ on the union of an equispaced grid of $10000$ points in $[0,1]$ and the $10000$ points {\tt logspace(-30,0,10000)}, which also covers the region below the smallest sample point.

Figure~\ref{fig:rat_approx} shows the results. We compare the approximants through the degree $N$ in the rate $\exp(-\pi\sqrt{2N})$, with $N = N_1 + N_2$ for the lightning approximant, of type $(N_1+N_2,N_1)$, and $N = m-1$ for the AAA approximant, of type $(m-1,m-1)$. The lightning error follows this rate up to about $N = 70$ and then levels off near $10^{-13}$. Except for the smallest degree $N = 3$, the AAA approximant is more accurate by roughly one to two orders of magnitude up to $N = 63$, where it reaches $8.6\cdot 10^{-14}$. With the uniformly clustered poles~\eqref{eq:uniformpoles} and $N_2 = 0$, the approximant converges more slowly, with an error that follows the rate $\exp(-\pi\sqrt{0.3N})$~\cite{HerremansHuybrechsTrefethen2023}.}

\begin{figure}
    \centering
\definecolor{clightning}{rgb}{0.00,0.20,0.60}
\definecolor{caaa}{rgb}{0.95,0.60,0.00}
\definecolor{cchebfun}{rgb}{0.70,0.10,0.10}
\definecolor{cuniform}{rgb}{0.10,0.50,0.20}
\begin{tikzpicture}
\begin{axis}[
    ymode=log,
    scale only axis,
    width=0.6\linewidth, height=0.5\linewidth,
    grid=major,
    xmin=1.5, xmax=10,
    ymin=1e-16, ymax=1,
    xtick={2,3,...,10},
    ytick={1e0,1e-5,1e-10,1e-15},
    minor ytick={1e-1,1e-2,1e-3,1e-4,1e-6,1e-7,1e-8,1e-9,1e-11,1e-12,1e-13,1e-14},
    xlabel={$\sqrt{N}$},
    ylabel={$\|f-r\|_\infty$},
    legend cell align=left,
    legend style={at={(0.02,0.02)}, anchor=south west, font=\small},
    table/x=sqrtN,
    clip=false,
]
\addplot[clightning, solid, mark=*]
    table[y=max_lightning] {rat_approx_conv.dat};
\addlegendentry{lightning, $N = N_1 + N_2$}
\addplot[caaa, dashed, mark=square*, mark options={solid}]
    table[y=max_aaa] {rat_approx_conv.dat};
\addlegendentry{AAA, $N = m - 1$}
\addplot[cchebfun, densely dashdotted, mark=triangle, mark options={solid}]
    table[y=max_chebfun] {rat_approx_chebfun.dat};
\addlegendentry{Chebfun \texttt{aaa}}
\addplot[cuniform, solid, mark=diamond]
    table[y=max_uniform] {rat_approx_conv.dat};
\addlegendentry{uniform only, $N = N_1$}
\addplot[black, densely dotted, forget plot, domain=1.732:9.85, samples=2]
    {1000*exp(-pi*sqrt(2)*x)}
    node[pos=0.97, anchor=north east, xshift=-3pt, yshift=13pt, font=\footnotesize, text=black] {$C\exp(-\pi\sqrt{2N})$};
\addplot[gray, densely dashed, forget plot, domain=1.732:10, samples=2]
    {0.3*exp(-pi*sqrt(0.3)*x)}
    node[pos=0.72, anchor=south west, yshift=-1pt, font=\footnotesize, text=black] {$C\exp(-\pi\sqrt{0.3N})$};
\end{axis}
\end{tikzpicture}
    \caption{Rational approximation of $\sqrt{z}$ on $[0,1]$ with the lightning and AAA approximants computed in our framework. Maximal error $\|f-r\|_\infty$ for increasing degree $N$. Chebfun \texttt{aaa}~\cite{Trefethen2014} is run on the same sample points with default settings. The curve labeled uniform only uses the uniformly clustered poles~\eqref{eq:uniformpoles} with $N_2 = 0$.}
    \label{fig:rat_approx}
\end{figure}

\section{Conclusion and future work}\label{sec:conclusion}
In this paper we have shown that computing an orthonormal basis of
vector-valued rational functions having $k$ components with respect to a discrete inner product is equivalent
to an inverse eigenvalue problem concerning a pencil of $k$-Hessenberg matrices.
Two algorithms to solve this problem were presented: an updating algorithm
and a rational Krylov-type algorithm.
The numerical experiments indicate that the updating algorithm leads to more
accurate results compared to the rational Krylov-type method.
Finally, we showed how the lightning and AAA approximants are obtained in this framework,
and considered the approximation of $\sqrt{z}$ on $[0,1]$.
For future work, it would be interesting to investigate how a downdating procedure—i.e., removing a node and its corresponding weight—might be designed. As noted at the end of Section~\ref{sec:updating}, one could develop an algorithm that manipulates the factorization of $K_k$ using techniques similar to those in \cite{MachVanBaVand2014}. Such an approach may improve both numerical stability and computational efficiency.

\section*{Acknowledgements}
Special thanks go to Marc Van Barel for the many useful discussions and advice on the \texttt{MATLAB} implementation. Furthermore, I would like to thank Raf Vandebril and Kobe Bergmans for carefully reading an initial draft of this manuscript. \rev{I also thank the anonymous reviewers for their valuable comments, which improved this manuscript.}

\section*{Funding}
This work was supported by the Fund for Scientific Research Flanders (FWO), G0B0123N (Short recurrence relations for rational Krylov and orthogonal rational functions inspired by modified
moments). 

\bibliographystyle{elsarticle-num} 
\bibliography{references.bib}






\end{document}